\documentclass[11pt,reqno]{amsart}
\usepackage{amsgen, amstext,amsbsy,amsopn, amsthm, amsfonts,amssymb,amscd,amsmath,euscript,enumerate,url,verbatim,calc,xypic,tikz}
 \usepackage{xcolor}
\usepackage{amsgen, amstext,amsbsy,amsopn, amsthm, amsfonts,amssymb,amscd,amsmath,euscript,enumerate,url,verbatim,calc,xypic,tikz}
\usepackage{amsfonts}
\usepackage{amssymb} 
\usepackage{graphicx}
\usepackage{mathrsfs}  
\usepackage{hyperref}
\usepackage{MnSymbol}
\usepackage{pgfkeys}
\usepackage{mathtools}
\usepackage{hyperref}
\usepackage{MnSymbol}
\usepackage{pgfkeys}
\usepackage{mathtools}
\def\multiset#1#2{\ensuremath{\left(\kern-.3em\left(\genfrac{}{}{0pt}{}{#1}{#2}\right)\kern-.3em\right)}}

\usetikzlibrary{arrows}

\usepackage{latexsym}
\usepackage{graphics}
\usepackage{color}
\usepackage{lastpage}
\usepackage{fancyhdr}
\usepackage{multirow}
\allowdisplaybreaks
\usepackage{graphicx}
\graphicspath{ {F:/IMAGES/} }

\newcommand{\m}{\mathfrak{m} }

\newcommand{\Ass}{\operatorname{Ass}}

 \newcommand{\lcm}{\operatorname{lcm}}

\newcommand{\sdefect}{\operatorname{sdefect}}
\newcommand{\sdefectideal}{\operatorname{sdefectideal}}

\newcommand{\proset}{\,\mathrel{\lower 4pt\hbox{$\scriptscriptstyle/$}
		\mkern -14mu\subseteq }\,} 

\newtheorem{theorem}{Theorem}[section]
\newtheorem{corollary}[theorem]{Corollary}
\newtheorem{lemma}[theorem]{Lemma}

\newtheorem{notation}[theorem]{Notation}
\usepackage{amsmath}

\theoremstyle{definition}

\newtheorem{remark}[theorem]{Remark}
\newtheorem{definition}[theorem]{Definition}

\newtheorem{example}[theorem]{Example}
\title[Symbolic defects of edge ideals  of unicyclic graphs] {Symbolic defects of edge ideals  of unicyclic graphs}
\author[ M. Mandal and D.K. Pradhan ]{Mousumi Mandal$^*$ and Dipak Kumar Pradhan}
 \thanks{$^*$ Supported by SERB(DST) grant No.: $\mbox{EMR}/2016/006997$, India}        
\thanks{AMS Classification 2010: 13F20, O5C25}
\address{Department of Mathematics, Indian Institute of Technology Kharagpur, 721302, India} \email{mousumi@maths.iitkgp.ac.in}
\address{Department of Mathematics, Indian Institute of Technology Kharagpur, 721302, India}\email{dipakkumar@iitkgp.ac.in}

\begin{document}
\maketitle
\begin{abstract}
 
 We introduce the concept of minimum edge cover for an induced subgraph in a graph.
 Let $G$ be a unicyclic graph with a unique odd cycle and $I=I(G)$ be its edge ideal. We compute the exact values of all symbolic defects of $I$ using the concept of minimum edge cover for an induced subgraph in a graph. We describe one method to find the quasi-polynomial associated with the symbolic defects of edge ideal $I$. We classify the class of unicyclic graphs  when some power of maximal ideal annihilates $ I^{(s)}/I^s $ for any fixed  $ s $. Also for those class of graphs, we compute the Hilbert function of the module $I^{(s)}/I^s$ for all $s.$  
 \vspace*{0.2cm}
\end{abstract}
\section{Introduction}  
Let $k$ be a field and $R=k[x_1,\ldots ,x_n]$ be a polynomial ring in $n$ variables. Let $I$ be a homogeneous ideal of $R$. 
Then for $s\geq 1$, the $s$-th symbolic power of $I$ is defined as $I^{(s)}=\displaystyle{\bigcap_{p\in \Ass I}(I^sR_p\cap R)}$. If minimal generators of $I$   are given, then we can easily understand the minimal generators of ordinary powers of $I,$ but it is still difficult many times to understand the minimal generators of symbolic powers of $I.$  Geometrically, one classical result of  Zairiski and Nagata says that the $s$-th symbolic power of a radical ideal consists of the elements that vanish up to order $s$ on the corresponding variety.  By definition, $I^s\subseteq I^{(s)}$ holds for all $s\geq 1$. The opposite containment is not true in general. In  \cite{ggsv}, F. Galetto et al.  have introduced an invariant known as the symbolic defect, that counts the number of generators which must be added to $I^s$ to make $I^{(s)}$.  The
 $ s $-th symbolic defect of $I$ is defined as $$ \sdefect(I,s) = \mu\left(I^{(s)}/I^s\right),$$ where $ \mu\left(I^{(s)}/I^s\right)$ denotes  the number of minimal generators of the module $ I^{(s)}/I^s.$ In other words, it is the measure of failure of $I^s$ to be  equal with $I^{(s)}.$ 
If $I^{(s)} \neq I^s ,$ J. Herzog 
asked the following questions  in \cite{herzog}:     
\begin{enumerate}
	\item  What is the number of generators of $ I^{(s)}/I^s ~   ?$
	
    \item What is the annihilator of $ I^{(s)}/I^s,$ and how does the exponent $\gamma(s)$ for which
    $ \m ^{ \gamma(s)}. (I^{(s)}/I^s) = 0 $ depend on $  s~   ?$    
	
\end{enumerate}  
In this paper, we are able to give complete answers of questions $(1)$ and  $(2)$ if $I$ is the edge ideal of a unicyclic graph with a unique odd cycle.
 In   \cite{ggsv}, F. Galetto et al. give an understanding of the symbolic defect of $I$ when $I$ is either the defining ideal of a star configuration or the ideal associated with a finite set of points in $ \mathbb{P}^2. $  If  $  I $ is an ideal whose symbolic Rees algebra is Noetherian, B. Drabkin and L. Guerrieri in   \cite{drabkin} prove that $ \sdefect(I,s) $ is eventually a quasi-polynomial as a function of $ s.$   They compute the symbolic defects explicitly for the cover ideals of complete graphs and odd cycles. Also, they find the quasi-polynomial associated with the symbolic defects of the cover ideal of a complete graph.
 
Our approach to compute $ \sdefect(I,s)$, is to find the number of minimal generators of ideal $ D(s) $ such that
$ I^{(s)} =  I^s + D(s) 
.$ If we consider $I$ to be the edge ideal of a bipartite graph $G,$ then by \cite[Theorem 5.9]{simis},  we have $I^{(s)} = I^s $ and hence $D(s)=0.$  But it is still  difficult to compute the number of minimal generators of $D(s)$ if $G$ contains an induced odd cycle.

Recently,  the  invariants associated with the symbolic powers of  edge ideals of some simple graphs  have been studied in \cite{bidwan,Y.GU,janssen}. M. Janssen et al. compute the $ \sdefect(I, s) $ for $n+1 \leq s \leq 2n+1,$ where $I$ is the edge ideal of an odd cycle $C_{2n+1}$ in \cite{janssen}. In  \cite{bidwan}, B. Chakraborty and M. Mandal
give an upper bound for the $ \sdefect(I, s) $ for $n+1 \leq s \leq m+1,$  where $I$ is the  edge ideal of a clique sum of two odd cycles    $ C_{2n+1}=(x_1,\ldots,x_{2n+1})   $ and  $ C_{2m+1}=(x_1,y_2,\ldots,y_{2n+1}) $  joined at  $x_1$ with $n  <   m$. In this paper, we give the exact value of $ \sdefect(I, s) $ for all $ s ,$ where $I$ is the edge ideal of a unicyclic graph with a unique odd cycle, using the concept of minimum edge cover for an induced subgraph in a graph.
\vspace*{0.1cm}\\
In \cite{arsie}, A. Arsie and J. E. Vatne compute the Hilbert function of the module $I^{(s)}/I^s,$ where $I$ is an ideal associated to  $n+1$ general points in $ \mathbb{P}^n .$ We compute the Hilbert function of the module $I^{(s)}/I^s$  where $I$ is the edge ideal of a unicyclic graph with a unique odd cycle such that some power of maximal ideal annihilates $ I^{(s)}/I^s $.
     
In Section 2, we recall all the definitions and results which will be required for the rest of the paper. In Section 3, we define the notion of minimum edge cover for an induced subgraph in a graph.  For $G$ to be a graph containing an induced odd cycle $C_{2n+1},$  we give the structure of each minimum edge cover for $C_{2n+1}$ in $G$. Also, we compute the number of possible edge sets of size $n+1+k$ for $k \geq 0$ containing a  minimum edge cover for $C_{2n+1}$ in $G$ (see Lemma \ref{type1} and Lemma \ref{type2}).  

In Section 4, we compute all the symbolic defects of the edge ideal of a unicyclic graph with a unique odd cycle in Theorem \ref{sd.unicyclic}. In Remark \ref{sd.unicyclic.2}, we give one procedure to find the quasi-polynomial associated with the symbolic defects of edge ideal of a unicyclic graph with a unique odd cycle.

In Section 5,  if $I$ is the edge ideal of a unicyclic graph with a unique odd cycle, we categorize the class of unicyclic graphs when some power of maximal ideal annihilates $ I^{(s)}/I^s$ in Theorem \ref{annhilator}.
Additionally, for the same class of graphs,  we compute the Hilbert function of the module $I^{(s)}/I^s$  in  Theorem \ref{hilbert.thm}.

\section{Preliminaries} 
In this section,  we recall some definitions and results that will be used throughout the paper. Let $G$ denotes a finite simple graph over the vertex set $V(G)$ and edge set $E(G).$
\begin{definition}
	 A subgraph $ H $ of $ G $ is called an induced subgraph if $ \{u,v\} \in E(G) $
implies $ \{u,v\} \in E(H)  $ for all vertices $ u $ and $ v $ of $ H. $ 
\end{definition}
\begin{definition}
For a vertex $ u $ in a graph $ G , $ let $ N_G(u) := \{v \in V(G) ~|~ \{u,v\} \in E(G)\} $ be the set of
neighbours of $ u   $ and set $N_G[u] := N_G(u)\cup \{u\}. $  For a subset of the vertices
$ W \subseteq  V (G),  $ $ N_G(W) $ and $ N_G[W] $ are defined similarly.   If $ |N_G(u)| = 1, $ then $ u $ is called a leaf of $ G.$ 
For a subset $W \subseteq V(G) $ of the vertices in $ G, $ define $ G \setminus  W $ to
be the subgraph of $ G $   with the vertices in $ W $ (and their incident edges) deleted.  If
$ W = \{u\} $ consists of a single vertex, we write $ G \setminus u $ instead of $ G \setminus \{u\}. $ For a vertex $ u $ in  $ G $, let  $ N_{E(G)}(u) := \{\{u,v\} \in E(G) ~|~ v \in N_G(u)\} $, i.e.,   $ N_{E(G)}(u)$  is the set of
all the edges of $G$ connected to the vertex $ u $. For a subset $Y \subseteq E(G) $ of the edges in $ G, $ define $ G \setminus  Y $ to
be the subgraph of $ G $ with the  edge set $E(G)  \setminus  Y$  and  vertex set  $V(G)  \setminus  V_Y$  where  $V_Y = \{u \in V(G)~|~  N_{E(G)}(u)  \subseteq  Y   \}$.       
\end{definition}        

\begin{definition} 
Let $G$ be a simple graph on $n(G)$ vertices.    
\begin{enumerate}
\item An edge cover of $ G  $ is a set $ L \subseteq E(G)$   such that every vertex of $ G $   is incident to some edge of $ L. $
\item  An  edge cover of $G$ is minimal if no proper subset of it is an edge cover of $G.$            
\item A minimum edge cover of $G$ is a minimal edge cover of $G$ of the smallest possible size.    
So every minimum edge cover is a minimal edge cover but the converse need not be true.
\item The size of a minimum edge cover in a graph $ G $ is known as the edge cover number of $ G, $ denoted by  $ \beta^\prime(G). $
\item A collection of edges $ \{e_1, \ldots , e_t\} \subseteq E(G)$ for some $t \geq 1$ is called a matching if they are pairwise
disjoint.  Let $ M $ be a matching of $ G. $  A set $ S $ of vertices
in $ G $ is saturated by $ M $ if every vertex of $ S $ is in an  edge of  $ M. $  A set $ U $ of vertices
in $ G $ is unsaturated by $ M $ if no vertex of $ U $ is in 
an edge of  $ M. $  The maximum matching in $G$ is a matching of largest size and it is denoted by
 $\alpha^\prime(G).$  By \cite[Theorem 3.1.22]{west},  it is  known that $   \beta^\prime(G) = n(G)- \alpha^\prime(G)  $.
\end{enumerate} 
\end{definition}

\begin{notation}
	The number of $k$-element multisets  on $ n $ elements is  termed as ``$ n $ multichoose $ k,$'' denoted by $ \displaystyle{   \multiset{n}{ k}       }  .$ Here $ n $ multichoose $ k$ is given by the simple formula $$ \displaystyle{   \multiset{n}{ k} = \binom{n+k-1}{k}      }  $$	
and we assume     $\displaystyle{\multiset{n}{k}=0}$ if $k < 0 $. 	
	
\end{notation}
Let $ k $ be a field and  $ R = k[x_1,\ldots, x_n] $ be a standard graded polynomial ring over
$ n $ variables. 	The	degree of a monomial $f \in  R$ is denoted by $\deg(f).$ 
\begin{definition}    
	
	Let $I \subset R$ be a monomial ideal. Let $ \mathcal{G}(I) $ be the set of minimal generators of the ideal $I.$		
	Let $D(s)$ be the ideal we 
	need to add to $I^s $ to achieve $ I^{(s)} ,$ i.e.,
	$ I^{(s)} = I^s + D(s).$
	\begin{enumerate}
		\item We set $ \sdefectideal (I^{(s)}) = D(s) .$ Then we can write $$ I^{(s)} = I^s + \sdefectideal (I^{(s)}) .$$
		\item We set    $ \sdefect(I,s) $ is the number of elements of $\mathcal{G}(D(s))$, i.e.,
		$$ \sdefect(I,s)= |\mathcal{G}(\sdefectideal (I^{(s)}) ) |.$$  
	\end{enumerate}
	
\end{definition}

\begin{definition}
The edge ideal of $ G $ is defined to be
$$ I(G) = \left\langle uv ~|~ \{u, v\} \in E(G)  \right\rangle  \subset  R .$$
If $L \subseteq E(G)$ is an edge set, then $\{ uv ~|~ \{u, v\} \in  L     \}$ is the set of minimal generators of $I(G)$ corresponding to that edge set.    
\end{definition}

In order to describe the symbolic defects of edge ideal of a unicyclic graph  and the quasi-polynomial associated  with the  symbolic defects, we  used the following two results.  
\begin{lemma}\cite[Theorem 3.4]{Y.GU}\label{unicyclic}
	Let $ G $ be a unicyclic graph with a unique cycle $ C_{2n+1} = (x_1, \ldots , x_{2n+1}), $ and
	let $ I = I(G) $ be its edge ideal. Let $ s \in \mathbb{N} $ and write $ s = k(n + 1) + r $ for some $ k \in \mathbb{Z} $ and
	$ 0 \leq r \leq n. $ Then   
	$${I}^{(s)} =   \displaystyle{ \sum_{t=0}^{k}  {I}^{s-t(n+1)}  (x_1 \cdots x_{2n+1})^t  }. $$ 
	
\end{lemma}       

\begin{lemma}\cite[Theorem 2.4]{drabkin}\label{quasi}  
Let $ R $ be a Noetherian local or graded-local ring with maximal (or homogeneous maximal) ideal $ \m $ and residue field $ R/\m = k. $ Let $ I $ be an ideal (or homogeneous ideal)
of $ R $ such that $ R_s(I) $ is a Noetherian ring, and let $ d_1,\ldots, d_r $ be the degrees of the generators
of $ R_s(I) $ as an $ R- $algebra. Then $ \sdefect(I, s) $ is eventually a quasi-polynomial in $ s $ with
quasi-period $ d =$    $\lcm(d_1,\ldots , d_r).  $     
\end{lemma}

\section{Edge cover  for an induced Subgraph in a Graph}     
In this section,  we define the minimum edge cover  for an   induced subgraph in a graph. 
\begin{definition} 
Let $G$ be a finite simple graph and $H$ be an  induced subgraph of $G$.
	\begin{enumerate}    
		\item An edge cover    for $H$ in $ G  $ is a set $ L_H \subseteq E(G)$  such that every vertex of $ H $   is incident to some edge of $ L_H$ in $G.$ Note that an edge cover    for $H$ in $ G  $ may contain some edge of $G$ which are not in  $E(H).$ If all the edges of $L_H$ lie in $E(H),$ we say that $ L_H $  is a Type-I  edge cover for $H$ in $G$ and if some edge of $ L_H $  lies in $E(G) \setminus E(H),$ we call it as a Type-II  edge cover for $H$ in $G.$  
	\item A Type-I (respectively Type-II) edge cover $L_H$ for $H$ in $ G  $ is minimal if no proper subset of $L_H$  is  a Type-I (respectively Type-II) edge cover for $H$ in $ G  .$  
\item  A Type-I (respectively Type-II) minimum edge cover for $H$ in $ G  $ is a Type-I (respectively Type-II) minimal edge cover for $H$ in $ G  $ of the smallest possible size.  
	So every Type-I (respectively Type-II) minimum edge cover for $H$ in $ G  $ is a Type-I (respectively Type-II) minimal edge cover for $H$ in $ G  $ but the converse may not be true.        
\item The size of a  Type-I (respectively Type-II) minimum edge cover for $H$ in $ G  $ is known as the Type-I (respectively Type-II) edge cover number for $H$ in $ G, $ denoted by  $ \gamma^\prime(H) $(respectively $ \gamma^{\prime\prime}(H) $).
    
\end{enumerate}
\end{definition}

\textbf{Minimum edge cover for an induced odd cycle in a graph:}  
 Let   $H = C_{2n+1}=(x_1,\ldots,x_{2n+1}) $ be an induced subgraph of $G.$  
Let $ e_j $ represent the  $ j $-th edge in the cycle, i.e., $  e_j = \{x_j,x_{j+1}\} $ for $ 1 \leq j \leq 2n  $
and $ e_{2n+1} = \{x_{2n+1},x_1\}.$
\vspace*{0.2cm}\\	
First, we give the structure of each Type-I minimum edge cover for $C_{2n+1}$ in $G$ in the following remark.    
\begin{remark}\label{remark1}
  Since $|E(C_{2n+1})| = 2n+1,$ $n$ alternate edges of $C_{2n+1}$ forms a maximum matching in $C_{2n+1},$ i.e., $ \alpha^\prime(C_{2n+1}) = n.$  Then $ \beta^\prime(C_{2n+1})=|V(C_{2n+1})|-\alpha^\prime(C_{2n+1})=(2n+1) - n = n+1  $.  Note that  every  minimum edge cover of  $C_{2n+1}$   contains some maximum matching $M$ in $C_{2n+1}$.	
Here every maximum matching $M$ in $C_{2n+1}$ saturates $2n$ vertices of $C_{2n+1}$ and one vertex remains unsaturated. Thus every  minimum edge cover of  $C_{2n+1}$  consists of edges  of a maximum matching $M$ in $C_{2n+1}$ and one of the two edges of $ C_{2n+1}$ connected with the  vertex of $C_{2n+1}$ not saturated by the edges of $M.$ 
By definition, every minimum edge cover of $C_{2n+1}$ is equivalent to a Type-I minimum edge cover for $C_{2n+1}$ in $G$ which implies $  \gamma^\prime(C_{2n+1}) = \beta^\prime(C_{2n+1}) = n+1. $ Hence every Type-I minimum edge cover for  $C_{2n+1}$ in $G$ consists of edges of a maximum matching $M$ in $C_{2n+1}$ and one of the two edges of $ C_{2n+1}$ connected with the  vertex of $C_{2n+1}$ not saturated by the edges of $M.$
\end{remark}
Next, we describe the structure of each Type-II minimum edge cover for $C_{2n+1}$ in $G$ in the following remark.
\begin{remark}\label{remark2}
 Let $M$ be a maximum matching in $C_{2n+1}.$ The $n$ edges of $M$ saturate $2n$ vertices  and only one vertex remains unsaturated.  We can construct a Type-II  edge cover for $C_{2n+1}$ in $G$ of size $n+1$ by   adding one edge from $E(G)\setminus E(C_{2n+1})$ to $M$ at the unsaturated vertex. Since any  $n$ edges of $G$ can not saturate $2n+1$ vertices of $C_{2n+1}$, $  \gamma^{\prime\prime}(C_{2n+1}) =  n+1. $    
Suppose a Type-II edge cover for $C_{2n+1}$ in $G$ does not contain any maximum matching $M$ of $C_{2n+1}.$
 Even if there is a Type-I edge cover for $C_{2n+1}$  in $G$ which uses only edges of $C_{2n+1}$ and does not contain any maximum matching $M$ of $C_{2n+1},$  its size is greater than $n+1.$  
Since $C_{2n+1}$ is an induced subgraph of $G$ and we want to use at least one edge of $E(G)\setminus E(C_{2n+1})$ in a Type-II edge cover for $C_{2n+1}$ in $G,$ 
 the number of edges of that Type-II edge cover for  $C_{2n+1}$ in $G$ must be greater than $n+1 $. Hence every Type-II minimum edge cover for  $C_{2n+1}$ in $G$ must contain some maximum matching $M$ in $C_{2n+1}$. We can say every Type-II minimum edge cover for  $C_{2n+1}$ in $G$ consists of edges of a maximum matching $M$ in $C_{2n+1}$ and one edge  from $E(G)\setminus E(C_{2n+1})$ connected with the  vertex of $C_{2n+1}$ not saturated by the edges of $M.$

\end{remark}  
By Remark \ref{remark1} and Remark \ref{remark2}, general form of Type-I and Type-II minimum edge cover   for $ C_{2n+1} $ in $ G $ are given below.
\begin{enumerate}
\item \textbf{Type-I minimum edge cover for $ C_{2n+1} $ in $  G $:}\\
A Type-I minimum edge cover  for $ C_{2n+1} $ in $  G $ is a edge set of size $ n+1 $  consists of   a maximum matching of $ n $ alternate edges     $ e_{j+1},e_{j+3},\ldots,e_{j+2n-1} $ mod $2n+1$ for some  $ j \in [2n+1]$ and any one edge between $e_j$ and $ e_{j+2n}$ mod $2n+1.$
  
\item \textbf{Type-II minimum edge cover($x_j$) for $ C_{2n+1} $ in $  G $:}\\
A Type-II minimum edge cover($x_j$)  for $ C_{2n+1} $ in $  G $ is a edge set of size $ n+1 $   consists of the maximum matching of   $ n $ alternate edges     $ e_{j+1},e_{j+3},\ldots,e_{j+2n-1} $  mod $2n+1$ for some  $ j \in [2n+1]$   and any one edge of the set $E(G) \setminus E(C_{2n+1})$ connected with $ x_j .$
\end{enumerate}        
  
\begin{example}
 Consider $G$ to be the whiskered cycle as in Figure \ref{fig1}. Here $V(G) = \{ x_1,\ldots , x_{10}\}$ and $E(G)= \{ e_1,e_2,e_3,e_4,e_5,f_1,f_2,f_3,f_4,f_5   \},$	where $ e_j = \{x_j,x_{j+1}\} $   for $ 1 \leq j \leq 4  ,$ $e_5 = \{x_5,x_1\}$ and     $ f_j = \{x_j,x_{j+5}\} $   for $ 1 \leq j \leq 5.$ Let $H = C_5= (x_1,\ldots,x_5) .$ Note that $C_5$ is an induced subgraph of $G.$ By Remark \ref{remark1} and Remark \ref{remark2}, we have $  \gamma^{\prime}(C_{5}) =  3   $ and $  \gamma^{\prime\prime}(C_{5}) =  3. $      	
 Below we give all the minimum edge covers  for $H= C_{5} $ in $  G .$
 \begin{enumerate}
 	\item Type-I minimum edge covers  for $H= C_{5} $ in $  G $ are $\{e_1,e_2,e_4\},$
 	$\{e_2,e_3,e_5\},$$\{e_3,e_4,e_1\},$\\$\{e_4,e_5,e_2\}$ and $\{e_5,e_1,e_3\}$.	
 	\item Type-II minimum edge covers  for $H= C_{5} $ in $  G $ are 
 	$\{e_2,e_4,f_1\},$$\{e_3,e_5,f_2\},$\\$\{e_4,e_1,f_3\},$ $\{e_5,e_2,f_4\}$ and $\{e_1,e_3,f_5\}$. In particular,
 	\begin{enumerate}
 		
 		\item  Type-II minimum edge cover($x_1$)  for $H= C_{5} $ in $  G $ is $\{e_2,e_4,f_1\},$	
 		\item  Type-II minimum edge cover($x_2$)  for $H= C_{5} $ in $  G $ is $\{e_3,e_5,f_2\},$  
 		\item  Type-II minimum edge cover($x_3$)  for $H= C_{5} $ in $  G $ is $\{e_4,e_1,f_3\}$,
 		\item  Type-II minimum edge cover($x_4$)  for $H= C_{5} $ in $  G $ is $\{e_5,e_2,f_4\}$,
 		\item  Type-II minimum edge cover($x_5$)  for $H= C_{5} $ in $  G $ is $\{e_1,e_3,f_5\}.$	
 	\end{enumerate}	
 	
 \end{enumerate}

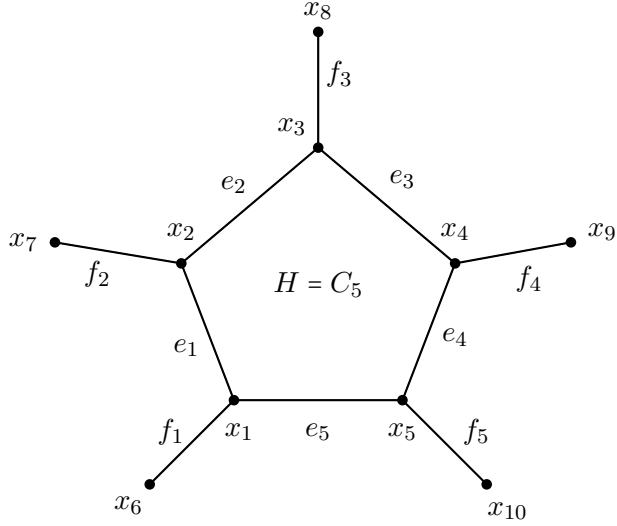
\begin{figure}[!ht]
	\begin{tikzpicture}[scale=1.4]
	\begin{scope}[ thick, every node/.style={sloped,allow upside down}] 
	\definecolor{ultramarine}{rgb}{0.07, 0.04, 0.56} 
	\definecolor{zaffre}{rgb}{0.0, 0.08, 0.66}   

	\draw[fill][-,thick] (-1,0) --(0.6,0); 
	\draw[fill][-,thick] (1.1,1.3) --(0.6,0); 
	\draw[fill][-,thick] (-1.5,1.3) --(-1,0);  
	\draw[fill][-,thick] (1.1,1.3) --(-0.2,2.4); 
	\draw[fill][-,thick] (-1.5,1.3) --(-0.2,2.4);
	\node at (-0.95,-0.3) {$x_1$};
	\node at (0.6,-0.3) {$x_5$};
	\node at (-1.5,1.6) {$x_{2}$};
	\node at (1.1,1.6) {$x_4$};
	\node at (-0.45,2.6) {$x_3$};           
	\draw[fill][-,thick] (-1,0) --(-1.8,-0.8);  
	\draw[fill][-,thick] (0.6,0) --(1.4,-0.8); 
	\draw[fill][-,thick] (-1.5,1.3) --(-2.7,1.5);	
	
	\draw[fill][-,thick] (1.1,1.3) --(2.2,1.5); 
	\draw[fill][-,thick] (-0.2,2.4) --(-0.2,3.5);	
	\draw [fill] [fill] (-1.8,-0.8) circle [radius=0.04];
	\draw [fill] [fill] (1.4,-0.8) circle [radius=0.04];
	\draw [fill] [fill] (-2.7,1.5) circle [radius=0.04];
	\draw [fill] [fill] (2.2,1.5) circle [radius=0.04];
	\draw [fill] [fill] (-0.2,3.5) circle [radius=0.04];
	\node at (-2,-1) {$x_6$};
	\node at (1.6,-1.05) {$x_{10}$};
	\node at (-3,1.5) {$x_{7}$};
	\node at (2.5,1.6) {$x_9$};
	\node at (-0.2,3.7) {$x_8$};  
	
	\node at (-1.45,0.5) {$e_1$};
	\node at (-0.2,-0.3) {$e_5$};
	\node at (-1,2.05) {$e_{2}$};
	\node at (1.1,0.6) {$e_4$};
	\node at (0.6,2.1) {$e_3$};         
	
	\node at (-1.6,-0.3) {$f_1$};
	\node at (1.3,-0.3) {$f_5$};
	\node at (-2.3,1.2) {$f_{2}$};
	\node at (1.8,1.15) {$f_4$};
	\node at (-0.0,3.1) {$f_3$};   
	
	\node at (-0.2,1.1) {$H=C_{5}$};

	\draw [fill] [fill] (-1,0) circle [radius=0.04];
	\draw [fill] [fill] (0.6,0) circle [radius=0.04];
	\draw [fill] [fill] (1.1,1.3) circle [radius=0.04];
	\draw [fill] [fill] (-0.2,2.4) circle [radius=0.04];
	\draw [fill] [fill] (-1.5,1.3) circle [radius=0.04];

	\end{scope}
	\end{tikzpicture}  
	\caption{Whiskered cycle $G=W(C_5)$.}\label{fig1}   	 
\end{figure}

\end{example}

\begin{notation}
Let $ G $   be a graph    containing an induced cycle  $ C_{2n+1}=(x_1,\ldots,x_{2n+1}). $ Note that $G$ may contain an  induced cycle other than $C_{2n+1}$, i.e., $G$ may not be a unicyclic graph. Let $ e_j $ denotes  the  $ j $-th edge in the cycle $C_{2n+1}$, i.e., $  e_j = \{x_j,x_{j+1}\} $   for $ 1 \leq j \leq 2n  $      
and $ e_{2n+1} = \{x_{2n+1},x_1\}.$ Let $G^{\prime}= G \setminus E(C_{2n+1})$ and   $ |E(G^\prime) | = l   $ which implies   $ |E(G)| = 2n+1 + l .  $  
\end{notation}

First, we want to compute the number of possible edge sets  containing  a Type-I minimum  edge cover for    $ C_{2n+1} $ in $G.$
\begin{lemma} \label{type1}    
  Let $ d^\prime_{n+1},d^\prime_{n+2},\ldots,d^\prime_{2n+1+l} $ be the number of     edge sets of  size $ n+1,n+2,\ldots,2n+1+l ,$ respectively       such that each edge set  contains  a Type-I minimum  edge cover for    $ C_{2n+1} $ in $G.$      Then $\displaystyle{ d^\prime_{n+1+k} = (2n+1)\binom{n-1 + l}{k}}$ for $  0\leq k \leq  n-1   ,$  \\ $\displaystyle{ d^\prime_{2n+1+k} = (2n+1)\binom{n-1 + l}{n+k}} + \binom{l}{k}$ for $  0 \leq k \leq l-1   $ and $  d^\prime_{2n+1+l} = 1. $            
\end{lemma}  	
\begin{proof}
 By definition, any Type-I minimum edge cover for $C_{2n+1}$ in $G$ consists of two consecutive edges and $n-1$ alternate edges of $C_{2n+1}$. There are  $ 2n + 1 $ such Type-I minimum edge covers  for $C_{2n+1}$  in  $  G $ and let us denote these by  $ E_1,E_2,E_3,\ldots,E_{2n+1}, $   where    $ E_1= \{e_1,e_2,e_4,e_6,\ldots,e_{2n-2},e_{2n} \},$  $E_2= \{e_2,e_3,e_5,e_7,\ldots,e_{2n-1},e_{2n+1} \} ,$  $ E_3=\{e_3,e_4,e_6,e_8,\ldots,\\e_{2n},e_1 \},$ \ldots $,E_{2n+1} = \{e_{2n+1},e_1,e_3,e_5,\ldots,e_{2n-3},e_{2n-1} \} $. 
	\vspace*{0.2cm}\\  	
	First, we want to find the  number of distinct edge sets of size $ n + 1 + k   $ containing  a Type-I  minimum edge cover for    $ C_{2n+1} $ in $G$  and at most $2n$ edges of $C_{2n+1}$ for $1 \leq k \leq  n-1+l.$ Thus we have to add  $ k $ edges  to each Type-I minimum edge cover for    $ C_{2n+1} $ in $G$. Without loss of generality we can start from the minimum edge cover $ E_1 .$         
	Let the set of $ k $  edges  consists of $ e_3 $ and some $ k-1 $ element set $ K $  where $ K \subseteq E(G) \setminus ( {E_1 \cup \{  e_3     \}}   ) .$  
	 After adding the edge  $ e_3 $ and edge set $ K   $ to $ E_1  ,$ the edge set of size  $ n+1+k $  will be  $ \{e_1,e_2,e_4,e_6,\ldots,e_{2n-2},e_{2n} \} \cup \{ e_3   \} \cup K .$ But we can get the same edge set by adding $ e_2 $ and $ K $ to $ E_3=\{e_3,e_4,e_6,e_8,\ldots,e_{2n},e_1 \}.$	   
Since we need distinct edge sets of size $ n+1+k,$
we do not include $ e_3 $ in the set of  $ k $  edges for $ E_1 $.
Thus for $E_1,$ we choose  the  set of $k$ edges to be a subset of $E(G) \setminus ( {E_1 \cup \{  e_3     \}}   ).$
Note that $e_3$ is the first missing edge after the two  consecutive edges $e_1$ and $e_2$ in $E_1$. Similarly,  we do not include the first missing edge after the  two consecutive edges for other $ E_i $'s, i.e., we do not add   $  e_4,e_6,\ldots,e_{2n},e_1  $ to $ E_2,E_4,\ldots,E_{2n-2},E_{2n} ,$ respectively.          Here each of $ E_3,E_5,\ldots,E_{2n+1} $ contains   $ e_3 .$
If   $P \subseteq  E(G) \setminus ( {E_1 \cup \{  e_3     \}}   )$ is a  set of $k$  edges,
then $ \{e_4,e_6,\ldots,e_{2n},e_1\}  \subset E_1 \cup P  $ and $e_3 \notin E_1 \cup P.$ 
So $E_1 \cup P$ is    distinct from the other edge sets of size $ n+1+k  $ we get,  after adding  $ k $  edges to other $ E_i   $'s by the same procedure.
Thus if we add  $ k $ edges  from $E(G) \setminus ( {E_1 \cup \{  e_3     \}}   )$ to $ E_1 $,  we get an   edge set of size $ n+1 +k $ which  is distinct from other edge sets of size $ n+1+k  $ we get,  after adding $ k $  edges to other $ E_i   $'s in the same procedure.  Hence the number of edge sets   of size  $ n+1+k $ we get, after adding  $ k $ edges  from the remaining $(2n+1+l)-(n+2)=n-1+l$ edges of $E(G) \setminus ( {E_1 \cup \{  e_3     \}}   )$ to $ E_1 ,$ is  $\displaystyle{   \binom{n-1+l}{k} }.$      Similarly, we can add $ k $ edges  to  other $ E_i $'s.  Note that in this technique of addition of  edges,  any edge set of size  $ n+1+k $ contains at most $2n$ edges of $C_{2n+1}.$  
	Since there are $ (2n+1) $ $ E_i $'s,   the number of edge sets of size  $ n+1+k $ containing one of $E_i$'s and at most $2n$ edges of $C_{2n+1}$ is   $\displaystyle{   (2n+1)\binom{n-1+l}{k}}$ for $  0\leq k \leq  n-1+l. $
	\vspace*{0.2cm}\\
	Next, we want to find the number of  edge sets of size $ 2n + 1 + k   $ containing  $2n+1$ edges of $C_{2n+1}$ for $0 \leq k \leq l.$ Since the  $2n+1$ edges of $C_{2n+1}$ are fixed, number of  edge sets of size $ 2n + 1 + k   $ containing  a Type-I  minimum edge cover for    $ C_{2n+1} $ in $G$  and  $2n+1$ edges of $C_{2n+1} $ is $\displaystyle{ \binom{l}{k} } $ for $0 \leq k \leq l.$
	\vspace*{0.2cm}\\   
	Hence $\displaystyle{ d^\prime_{n+1+k} = (2n+1)\binom{n-1 + l}{k}}$ for $  0\leq k \leq n-1   ,$   $\displaystyle{ d^\prime_{2n+1+k} = (2n+1)\binom{n-1 + l}{n+k}} + \binom{l}{k}$ for $ 0 \leq k \leq l-1   $ and $  d^\prime_{2n+1+l} = \displaystyle{ \binom{l}{l} } = 1. $   		
\end{proof}

Next, we want to compute the number of possible edge sets  containing  a Type-II minimum  edge cover but not any Type-I minimum edge cover for    $ C_{2n+1} $ in $G.$
\begin{notation}
Let   there are $ m  $
common vertices of $ C_{2n+1} $ and $G^\prime.$ 
Let    $ {G^\prime}_r = V(C_{2n+1}) \cap V(G^\prime) = \{ y_1,y_2,\ldots,y_m \}  $ where each $ y_i = x_j $ for some $j \in [2n+1]$. Let $ l_1,l_2,\ldots,l_m $      be the number of edges of   $ {G^\prime} $  connected to those $ m $ vertices  $ y_1,y_2,\ldots,y_m   ,$   respectively. Let $ u_i =  (n-1)   + (l-l_i)    $ for $1 \leq i \leq m$.

\end{notation}

\begin{remark}
	If an edge set   contains  a Type-II minimum    edge cover($x_j$) for some $j \in [2n+1]$   but not any Type-I minimum edge cover    for    $ C_{2n+1} $ in $G$, then that edge set can not include any of two edges $e_{j-1}$ and $ e_j$ mod $2n+1.$ Thus  at least two  edges of $C_{2n+1}$ are missing in any edge set  containing  a Type-II minimum    edge cover    but not any Type-I minimum edge cover    for    $ C_{2n+1} $ in $G$. Note that $|E(G)| = 2n+1+l.$ Hence there does not exist any edge set of size $2n+l$ or $2n+1+l$   containing  a Type-II minimum    edge cover   but not any Type-I minimum edge cover    for    $ C_{2n+1} $ in $G$.            
		
\end{remark}
\vspace*{-0.4cm}
\begin{lemma}\label{type2} 
    Let $ d^{\prime\prime}_{n+1},d^{\prime\prime}_{n+2},\ldots,d^{\prime\prime}_{2n-1+l} $ be the number of   edge sets of  size $ n+1,n+2,\ldots,2n-1+l ,$ respectively         such that     each edge set   contains  a Type-II minimum    edge cover   but not any Type-I minimum edge cover    for    $ C_{2n+1} $ in $G$. Then
  	\vspace*{0.2cm}\\
     $ d^{\prime\prime}_{n+k} = \displaystyle{\bigg[\binom{l_1}{1}\binom{u_1}{ k  - 1} + \binom{l_1}{2}\binom{u_1}{ k - 2 }}\displaystyle{ + \cdots + \binom{l_1}{l_1}\binom{u_1}{k - l_1}}\bigg]$
   	\vspace*{0.2cm}\\
   \hspace*{1.1cm}  $ +   
	$ $\cdots$ $  + \displaystyle{\bigg[\binom{l_m}{1}\binom{u_m}{ k   - 1}} + \binom{l_m}{2}\binom{u_m}{k  - 2}+ \cdots +\displaystyle{ \binom{l_m}{l_m}\binom{u_m}{k-l_m}  \bigg]}$   for $  1  \leq k \leq  n-1 + l  ,$
		\vspace*{0.2cm}\\
	 where we assume  $\displaystyle{\binom{i}{j} = 0}$ if $j > i$ or $j < 0 .$  		
	
\end{lemma}        	
\begin{proof}  	
	Here  the number of edges connected with  $ y_1 $ is $ l_1 $ in $ G^\prime $   and $ y_1 = x_j $ for some  $ j \in [2n+1].$ Let the $l_1$ edges connected with $ x_j $ are $f_1,\ldots,f_{l_1}.$  Then the only maximum matching saturating the $2n$ vertices of  $V(C_{2n+1}) \setminus \{x_j\}$ is the set $M_1$ of $ n $ alternate edges  $ e_{j+1},e_{j+3},\ldots,e_{j+2n-1} $ mod $2n+1.$   
	So any Type-II minimum edge cover($x_j$)   for $ C_{2n+1}  $  in $ G $ consists of edges of $M_1$ and one edge from the $l_1$ edges  $f_1,\ldots,f_{l_1}.$ We want to choose   edge sets of size $n+k$ for some $k \geq 1$ containing a Type-II minimum edge cover($x_j$) but not a Type-I minimum edge cover  for $ C_{2n+1}  $  in $ G .$ Since the edges of $M_1$ are fixed in such a  edge set of size $n+k$, we need to choose  only $k$ edges where the $k$ edges consist of at
	least one edge from those $l_1$ edges $f_1,\ldots,f_{l_1}$ and the remaining from the edge set $E(G) \setminus (M_1 \cup \{e_{j-1},e_j,f_1,\ldots,f_{l_1}  \})$ whose size is $(2n+1+l) - (n+2+l_1)=(n-1)+(l-l_1)=u_1.$ Note that if $ e_{j-1}$ or $e_j $ is in the chosen set of $k$ edges, then that edge set of size $ n + k $   contains a    Type-I minimum   edge cover for $ C_{2n+1}  $ in $ G .$    
	Thus the  number of edge sets of size $ n + k $  containing a   Type-II minimum   edge cover($x_j$) but
not any Type-I minimum edge cover for $ C_{2n+1}  $ in $ G $   is      
	$           \displaystyle{\binom{l_1}{1}\binom{u_1}{ k  - 1} + \binom{l_1}{2}\binom{u_1}{ k - 2 }+ \cdots +}\displaystyle{  \binom{l_1}{l_1}\binom{u_1}{k - l_1}} $    for $  1\leq k \leq n-1 + l,$ where 	  $\displaystyle{\binom{i}{j} = 0}$ if $j > i$ or $j < 0 .$     
		\vspace*{0.1cm}\\
		Similarly, we can find the  number of edge sets of size $ n + k $  containing a   Type-II minimum   edge cover($x_j$) but
	 not any Type-I minimum edge cover for $ C_{2n+1}  $ in $ G $ for other $ x_j $'s.

	  Hence   $ d^{\prime\prime}_{n+k} = \displaystyle{\bigg[\binom{l_1}{1}\binom{u_1}{ k  - 1} + \binom{l_1}{2}\binom{u_1}{ k - 2 } + \cdots + \binom{l_1}{l_1}\binom{u_1}{k - l_1}}\bigg]$ $ \displaystyle{  }$  
		\vspace*{0.15cm}\\    
    $ \hspace*{2.3cm}+\cdots + \displaystyle{\bigg[\binom{l_m}{1}\binom{u_m}{ k   - 1}}\displaystyle{ + \binom{l_m}{2}\binom{u_m}{k  - 2} + \cdots + \binom{l_m}{l_m}\binom{u_m}{k-l_m}  }\bigg]$   for $  1\leq k \leq n-1 + l  ,$
    	\vspace*{0.15cm}\\
     where 	  $\displaystyle{\binom{i}{j} = 0}$
 		if $j > i$ or $j < 0 .$
   \end{proof}    

\begin{corollary}\label{ucc3.}
	Let $ d_{n+1},d_{n+2},\ldots,d_{2n+1+l} $ be the number of   edge sets of  size $ n+1,n+2,\ldots,2n+1+l ,$ respectively such that each edge set contains a Type-I or Type-II minimum   edge cover for $C_{2n+1}$  in $ G. $      Then  $ d_{n+k} =    d^\prime_{n+k} + d^{\prime\prime}_{n+k}   $ for $  1\leq k \leq n-1 + l  ,$           $ d_{2n+l} =    d^\prime_{2n+l}  $ and $ d_{2n+1+l} =    d^\prime_{2n+1+l}  .$   
	
\end{corollary}
\begin{proof}
	It follows directly from Lemma \ref{type1} and Lemma \ref{type2}.  	
\end{proof}

\section{Symbolic Defects of edge ideals of Unicyclic Graphs}
In this section, we compute the exact values of all symbolic defects of the edge ideal of a unicyclic graph with a unique odd cycle using the concept of minimum edge cover for an induced subgraph in a graph. In the next  lemma, we describe the minimal generators of $I^s$ for all $s \geq 1,$ where $I$ is the edge ideal of $  C_{2n+1}=(x_1,\ldots,x_{2n+1}) .$

\begin{lemma}\label{unique1}
	Let $I=  ( g_1, \ldots , g_{2n+1}  )$ be the edge ideal of the cycle $C_{2n+1}=(x_1,\ldots,x_{2n+1}),$ where  $  g_j = x_jx_{j+1} $   for $ 1 \leq j \leq 2n  $      
	and $ g_{2n+1} = x_{2n+1}x_1.$
	Then
	every minimal  generator $ M  $ of $I^s$ has a unique expression of the
	form $ M = g_1^{\alpha_1} \cdots  g_{2n+1}^{\alpha_{2n+1}}  $ where $ \displaystyle{ \sum_{i=1}^{2n+1} \alpha_i =  s }. $      
\end{lemma}    
\begin{proof}
 Suppose $ M = g_1^{\beta_1} \cdots  g_{2n+1}^{\beta_{2n+1}} $ is another expression.    
	Since each $ g_j $ has  degree $2,$ $ \displaystyle{ \sum_{i=1}^{2n+1} \alpha_i =  \sum_{i=1}^{2n+1} \beta_i } .$
	 The exponent of $ x_{2i }$ in $ M $
	is $  \alpha_{2i-1} + \alpha_{2i} = \beta_{2i-1} + \beta_{2i}    $ for $ 1 \leq i \leq n    . $  Therefore $ \displaystyle{ \sum_{i=1}^{2n} \alpha_i =  \sum_{i=1}^{2n} \beta_i } .$ Thus  $   \alpha_{2n+1} =  \beta_{2n+1}    .$  Due to the symmetry of the graph, we have  $   \alpha_{i} =  \beta_{i}   $  for $ 1 \leq i \leq 2n .     $
\end{proof}
  
Let $ G $     be a unicyclic graph     with a unique  cycle  $ C_{2n+1}=(x_1,\ldots,x_{2n+1}). $
Note that  $G  \setminus  E(C_{2n+1})$ is a forest. Let $T = G  \setminus  E(C_{2n+1})$  and  $ |E(T)| = l .$ This implies  $ |E(G)| = 2n+1 + l .  $  Suppose   there are $ m  $ vertices of $ C_{2n+1} $  to which some trees are connected. Let    $ T_r = V(C_{2n+1}) \cap V(T)  = \{ y_1,y_2,\ldots,y_m \}  $ where each $ y_i = x_j $ for some  $ j \in [2n+1] $.   Let $ l_1,l_2,\ldots,l_m $      be the number of edges connected to those $ m $ roots  $ y_1,y_2,\ldots,y_m   $ in   $ T,  $   respectively. Let $ u_i = (l-l_i) + (n-1)  $ for $1 \leq i \leq m$. We use these notations for the remainder of this section. 
The next  lemma describes the minimal generators of $I^s$ for all $s \geq 1,$ where $I$ is the edge ideal of $  G .$

\begin{lemma}\label{unique2}
Let $ I = I(G) $ be   the edge ideal of $G$. Let $ g_1,\ldots, g_q $ be the minimal 
 generators of $ I. $
Then for every $ s \geq  1 $ and non-negative integers $ \alpha_1,\ldots, \alpha_q $ with $\displaystyle{\sum_{i=1}^{q} \alpha_i = s},$ the monomial $ g_1^{\alpha_1}\cdots g_q^{\alpha_q} $ is a unique minimal generator of $I^s.$  	
	
\end{lemma}

\begin{proof}
	Let  $ E(T) = E(G) \setminus  E(C_{2n+1}) = \{f_1,\ldots, f_l\}. $
	We apply induction on $l$.
	\vspace*{0.2cm}\\
	Base case: $l=1.$ There is a leaf $ z_1 $ in $ G $ such that $   N_G(z_1) =  \{x_i  \} $ for some $x_i \in V(C_{2n+1})$. Without loss of generality let $x_i = x_1$.      
Then $I = ( g_1, \ldots , g_{2n+1} , g_{2n+2} ) ,$ where  $  g_j = x_jx_{j+1} $   for $ 1 \leq j \leq 2n  $, $ g_{2n+1} = x_{2n+1}x_1$ and $  g_{2n+2} = x_1z_1.$   
	We want to prove that
	every minimal  generator $ M  $ of $I^s$ has a unique expression of the
	form $ M = g_1^{\alpha_1} \cdots  g_{2n+1}^{\alpha_{2n+1}} g_{2n+2}^{\alpha_{2n+2}} .$
	\vspace*{0.2cm}\\
	Suppose $ M = g_1^{\beta_1} \cdots  g_{2n+1}^{\beta_{2n+1}}g_{2n+2}^{\beta_{2n+2}} $ is another expression. Since each $ g_j $ is of  degree $2,$ $ \displaystyle{ \sum_{i=1}^{2n+2} \alpha_i =  \sum_{i=1}^{2n+2} \beta_i } $
	. The exponent of $ z_{1}$ in $ M $
	is $  \alpha_{2n+2} = \beta_{2n+2}   .$ Therefore  $ \displaystyle{ \sum_{i=1}^{2n+1} \alpha_i =  \sum_{i=1}^{2n+1} \beta_i } .$ 
	Here    $  g_1^{\alpha_1} \cdots  g_{2n+1}^{\alpha_{2n+1}}  = g_1^{\beta_1} \cdots  g_{2n+1}^{\beta_{2n+1}}$ is a minimal generator of  $I(C_{2n+1})^{s- \alpha_{2n+2}}$. Thus by  Lemma \ref{unique1}, $   \alpha_{i} =  \beta_{i}   $  for $ 1 \leq i \leq 2n + 1.     $
\vspace*{0.1cm}\\      Assume that  $  l \geq  2.  $  There is a leaf $ z_p $ in $ G $ such that $  N_G(z_p) =  \{z_q\} $ for some $z_p\in V(G)\setminus V(C_{2n+1})$ and $z_q \in V(G).$ Let $ G^\prime = G \setminus z_p. $
Let $I = ( g_1, \ldots , g_{2n+1} , g_{2n+2} , \ldots , g_{2n+1+l}  ) ,$ where  $  g_j = x_jx_{j+1} $   for $ 1 \leq j \leq 2n  $, $ g_{2n+1} = x_{2n+1}x_1$ and  $  g_{2n+2} , \ldots , g_{2n+1+l}  $ are the minimal generators of $I$ corresponding to the edges   $ f_1,\ldots, f_l. $  Without loss of generality let $g_{2n+1+l}  = z_q z_p.$  We want to prove that
	every minimal  generator $ M  $ of $I^s$ has a unique expression of the
	form $ M = g_1^{\alpha_1} \cdots  g_{2n+1}^{\alpha_{2n+1}} g_{2n+2}^{\alpha_{2n+2}} \cdots g_{2n+1+l}^{\alpha_{2n+1+l}}.$
\vspace*{0.1cm}\\  
Suppose $ M = g_1^{\beta_1} \cdots  g_{2n+1}^{\beta_{2n+1}}g_{2n+2}^{\beta_{2n+2}} \cdots g_{2n+1+l}^{\beta_{2n+1+l}}$ is another expression.    
Since each $ g_j $ has degree $2,$ $ \displaystyle{ \sum_{i=1}^{2n+1+l} \alpha_i =  \sum_{i=1}^{2n+1+l} \beta_i } .$ The exponent of $ z_{p}$ in $ M $
	is $  \alpha_{2n+1+l} = \beta_{2n+1+l}   .$ Therefore  $ \displaystyle{ \sum_{i=1}^{2n+l} \alpha_i =  \sum_{i=1}^{2n+l} \beta_i } .$ 
	Here    $  g_1^{\alpha_1} \cdots  g_{2n+1}^{\alpha_{2n+1}} g_{2n+2}^{\alpha_{2n+2}} \cdots g_{2n+l}^{\alpha_{2n+l}}  = g_1^{\beta_1} \cdots  g_{2n+1}^{\beta_{2n+1}}g_{2n+2}^{\beta_{2n+2}} \cdots g_{2n+l}^{\beta_{2n+l}}  $ is a minimal generator of  $I(G^\prime )^{s- \alpha_{2n+1+l} }$. Thus by  induction hypothesis,  $\alpha_{i} =  \beta_{i}$    for $ 1 \leq i \leq 2n + l.     $  
\end{proof}
    
\begin{remark}\label{unique3}
	Let $ I = I(G) $ be   the edge ideal of $G$. Here $|E(G)| = 2n+1 + l,$ i.e., $|\displaystyle{\mathcal{G}(I)| = 2n+1+l}$. Hence by Lemma \ref{unique2} for all $s \geq 1,$  
 $$|\displaystyle{\mathcal{G}(I^s)| = \multiset{2n+1+l}{s} = \binom{2n+l+s}{s} }  .$$     	
	
\end{remark}

The next three lemmas are very important to prove our main result in Theorem \ref{sd.unicyclic}.

\begin{lemma}\label{cont.1}
	Let $ I = I(G) $ be   the edge ideal of $G$. Let $ s \in   \mathbb{N}  $ and write $ s = k(n + 1) + r  $ for some $ k \in \mathbb{Z}  $ and
$ 0 \leq  r    \leq  n. $  Let $c = x_1\cdots x_{2n+1}.$   Then
 $$   I^{s-(n+1)} \cap  I^{s-2(n+1)}(c) = I^{s-2(n+1)} (I^{n+1}\cap (c)) .$$

\end{lemma}

\begin{proof}
Here $  ``\supseteq "$  is obvious. Now we prove the other containment.
\vspace*{0.1cm}\\  
Suppose $g \in \mathcal{G}(I^{s-(n+1)}) \cap  I^{s-2(n+1)}(c) $. Here $ \deg(g) = 2 (s-(n+1))= 2 (s-2(n+1)) +   (2n+1)  + 1.$ So    
   $g=hcy_{i_1}$ for some $h     \in  \mathcal{G}(I^{s-2(n+1)})$ and $y_{i_1}$ $\in V(G).$
  \vspace*{0.1cm}\\
  \textbf{Case (1)} Assume that  $y_{i_1}    \in N[V(C_{2n+1})]. $\\
   Note that $cy_{i_1}   $ can be expressed as the product of exactly $n+1$ minimal generators of $I$. Thus  $ cy_{i_1} \in  \mathcal{G}(I^{n+1}).$ Also $ cy_{i_1} \in  \mathcal{G}(I^{n+1})\cap (c).$ Hence 
   $g=h(cy_{i_1})\in I^{s-2(n+1)} (I^{n+1}\cap (c)).$  
  \vspace*{0.1cm}\\ 
  \textbf{Case (2)} Assume that  $y_{i_1}    \notin N[V(C_{2n+1})]. $\\ Here $g=hcy_{i_1}$ and $h     \in  \mathcal{G}(I^{s-2(n+1)})$. 
Suppose $y_{i_1} $ is in some tree $T_1$ of the forest $T$ where $V(C_{2n+1}) \cap V(T_1) = \{x_i\}$ for some $i\in  [2n+1].$ As $T_1$ is a tree, there is only one path $P$ from $y_{i_1} $ to $x_i.$ Let $P = y_{i_1} y_{i_2}\ldots y_{i_t} x_i$ be that path whose length is $t$. By Lemma \ref{unique2}, $h$ has a unique representation. Let $h = g_1 g_2\cdots g_{s-2(n+1)}$ where each $g_i \in \mathcal{G}(I)$ and all $g_i$'s may not be distinct. Let $M(h)$ be the collection of $g_i$'s for $1 \leq i \leq {s-2(n+1)}.$  Then $g=hcy_{i_1} = g_1 g_2\cdots g_{s-2(n+1)} cy_{i_1}.$
\vspace*{0.1cm}\\  
$\textbf{\underline{\mbox{Case  (2.a)}}}$ Assume that the  path $P$ is  of odd length. Then $t$ is odd.\\
 Suppose $y_{i_2} y_{i_3},y_{i_4} y_{i_5}, \ldots, y_{i_{t-1}} y_{i_t}$ are present in    the collection  $ M(h) .$ Without loss of generality let   they are $ g_1, g_2,\ldots, g_{\frac{t-1}{2}}.$ Then
 \begin{align*}
 g&=hcy_{i_1}\\
  &= g_1 g_2\cdots g_{s-2(n+1)} cy_{i_1}\\
  &= g_1 g_2\cdots g_{\frac{t-1}{2}}g_{\frac{t-1}{2}+1} \cdots g_{s-2(n+1)} cy_{i_1}\\
  &= g_{\frac{t-1}{2}+1} \cdots g_{s-2(n+1)} y_{i_1} g_1 g_2\cdots g_{\frac{t-1}{2}}c\\
  &= g_{\frac{t-1}{2}+1} \cdots g_{s-2(n+1)} y_{i_1} (y_{i_2} y_{i_3})(y_{i_4} y_{i_5}) \cdots  (y_{i_{t-1}} y_{i_t}) c\\
  &= g_{\frac{t-1}{2}+1} \cdots g_{s-2(n+1)}( y_{i_1} y_{i_2}) (y_{i_3}y_{i_4})  \cdots ( y_{i_{t-2}} y_{i_{t-1}}) cy_{i_t} .    
 \end{align*} 
 Here $ g_{\frac{t-1}{2}+1} \cdots g_{s-2(n+1)}( y_{i_1} y_{i_2}) (y_{i_3}y_{i_4})  \cdots ( y_{i_{t-2}} y_{i_{t-1}}) \in I^{s -2(n+1)- \frac{t-1}{2} }I^{ \frac{t-1}{2} } = I^{s-2(n+1)}.$ Since  $y_{i_t}    \in N[V(C_{2n+1})], $   $ cy_{i_t} \in  \mathcal{G}(I^{n+1})$    by  \textbf{Case (1)} and $g  \in I^{s-2(n+1)} (I^{n+1}\cap (c)).$
\vspace*{0.1cm}\\
 Suppose  $y_{i_2} y_{i_3}$ is absent in  the  collection  $ M(h) .$ Since $g \in \mathcal{G}({I}^{s-(n+1)})$, there exists some minimal generator  $g_j=y_{j_1} y_{j_2}$ of $I$ where $y_{j_1} \in   N_G(y_{i_1})$ such that  
  $y_{i_1}$  pair up  with   $y_{j_1} $ to get some element of $\mathcal{G}(I)$ and $ y_{j_2}$ remains unpaired. If there is some minimal generator  $g_k=y_{j_3} y_{j_4}$ of $I$ where $y_{j_3} \in   N_G(y_{j_2}),$  then $y_{j_2}$ can pair up   with
 $y_{j_3} $ to get some element of $\mathcal{G}(I)$ and 
  $ y_{j_4}$ again remains unpaired.  We continue this process. Since
   $G$ is a unicyclic graph, 
   there is no  path from  $y_{i_1}$ to any vertex of $N[V(C_{2n+1})]$ along   $y_{j_1}$. Thus if $y_{i_2} y_{i_3}$ is absent in    the collection  $ M(h) ,$ then one variable always remains unpaired in $g.$ Since $g \in \mathcal{G}(I^{s-(n+1)})  ,$ its a contradiction.
 
   By the same argument, if any one of  $y_{i_4} y_{i_5},y_{i_6} y_{i_7},\ldots,y_{i_{t-1}} y_{i_t}$ is absent in    the collection  $ M(h) ,$ then one variable remains unpaired in $g.$     Since $g \in \mathcal{G}(I^{s-(n+1)})  ,$  its a contradiction.
     $\textbf{\underline{\mbox{Case  (2.b)}}}$   Assume that the  path $P$ is  of even length. Then $t$ is even.\\
  Suppose $y_{i_2} y_{i_3},y_{i_4} y_{i_5}, \ldots,y_{i_{t-2}} y_{i_{t-1}}, y_{i_{t}} x_{i}$ are present in    the collection  $ M(h) .$ Without loss of generality let they are $ g_1, g_2,\ldots,g_{\frac{t}{2}-1}, g_{\frac{t}{2}}.$ Then    
  \begin{align*}
  g&=hcy_{i_1}\\
  &= g_1 g_2\cdots g_{s-2(n+1)} cy_{i_1}\\
  &= g_1 g_2\cdots g_{\frac{t}{2}}g_{\frac{t}{2}+1} \cdots g_{s-2(n+1)} cy_{i_1}\\  
  &= g_{\frac{t}{2}+1} \cdots g_{s-2(n+1)} y_{i_1} g_1 g_2\cdots g_{\frac{t}{2}-1} g_{\frac{t}{2}}c\\
  &= g_{\frac{t}{2}+1} \cdots g_{s-2(n+1)} y_{i_1} (y_{i_2} y_{i_3})(y_{i_4} y_{i_5}) \cdots  (y_{i_{t-2}} y_{i_{t-1}})(y_{i_{t}} x_{i}) c\\    
  &= g_{\frac{t}{2}+1} \cdots g_{s-2(n+1)}( y_{i_1} y_{i_2}) (y_{i_3}y_{i_4})  \cdots ( y_{i_{t-1}} y_{i_{t}}) cx_{i} .    
  \end{align*} 
  Here $ g_{\frac{t}{2}+1} \cdots g_{s-2(n+1)}( y_{i_1} y_{i_2}) (y_{i_3}y_{i_4})  \cdots ( y_{i_{t-1}} y_{i_{t}}) \in I^{s -2(n+1)- \frac{t}{2} }I^{ \frac{t}{2} } = I^{s-2(n+1)}.$ Since  $x_{i}    \in N[V(C_{2n+1})], $  $ cx_{i} \in  \mathcal{G}(I^{n+1})$    by  \textbf{Case (1)} and $g  \in I^{s-2(n+1)} (I^{n+1}\cap (c)).$
 \vspace*{0.1cm}\\
 By the same argument as in $\textbf{\underline{\mbox{Case  (2.a)}}}$, if any one of  $y_{i_2} y_{i_3},y_{i_4} y_{i_5},\ldots,y_{i_{t}} x_{i}$ is absent in    the collection  $ M(h) ,$  we get a contradiction.                   
\end{proof}  
\begin{lemma}\label{uc1}      
	Let $ I = I(G) $ be   the edge ideal of $G$. Let $ s \in   \mathbb{N}  $ and write $ s = k(n + 1) + r  $ for some $ k \in \mathbb{Z}  $ and
	$ 0 \leq  r    \leq  n. $ Let $c = x_1\cdots x_{2n+1}.$ Then  for $2 \leq t \leq k-1 ,$ 
	$$  \mathcal{G}({I}^{s-(n+1)}) \cap  {I}^{s-(n+1)-t(n+1)}(c)^{t}   \subseteq {I}^{s-2(n+1)}(c).$$
\end{lemma}
\begin{proof}
	First we want to prove that	for $2 \leq t \leq k-1 ,$
	$$  \mathcal{G}({I}^{s-(n+1)}) \cap  {I}^{s-(n+1)-t(n+1)}(c)^{t}   \subseteq I^{s-(n+1)-(t-1)(n+1)}(c)^{t-1}.$$ 
	Let $g \in  \mathcal{G}({I}^{s-(n+1)}) \cap  {I}^{s-(n+1)-t(n+1)}(c)^{t}$ for some $t$ where $2 \leq t \leq k-1 .$ Here $ \deg(g) = 2 (s-(n+1))= 2 (s-(n+1)-t(n+1)) +   t(2n+1)  + t.$ So $g = hc^t{y_{i_1}\cdots   y_{i_{t}}}$ for some $h \in \mathcal{G}({I}^{s-(n+1)-t(n+1)})$ and $y_{i_j} \in V(G)$ for $1 \leq j \leq t.$
	\vspace*{0.1cm}\\ 
	\textbf{Case (1)}  Assume that there exists one  $y_{i_j}$ $ \in   N[V(C_{2n+1})]$ for some $j\in [t].$    
	\vspace*{0.1cm}\\
	Without loss of generality let it be  $y_{i_1}$. 
	Then $ g = hc^t{y_{i_1}\cdots   y_{i_{t}}}
	=h(c{y_{i_{1}}})(c^{t-1} ){{y_{i_2}}{y_{i_{3 }}}\cdots y_{i_{t}}} .$\\ 
	Note that $cy_{i_{1}}$  $\in \mathcal{G}(I^{n+1}).$ Hence
	$ 	g \in I^{s-(n+1)-t(n+1)}{I^{n+1}}(c)^{t-1}={I}^{s-(n+1)-(t-1)(n+1)}(c)^{t-1} .$ 
	\vspace*{0.1cm}\\
	\textbf{Case (2)} Assume that  no $y_{i_j} \in   N[V(C_{2n+1})]$  where $1 \leq j \leq t.$
	\vspace*{0.1cm}\\
	Then   
	$g = hc^t{y_{i_1}\cdots   y_{i_{t}}}$ for some $h \in \mathcal{G}({I}^{s-(n+1)-t(n+1)})$ and $y_{i_j} \in  V(G) \setminus N[V(C_{2n+1})]$  for $1 \leq j \leq t.$ By Lemma \ref{unique2}, $h$ has a unique representation. Let $h = g_1 g_2\cdots g_{s-(n+1)-t(n+1)}$ where each $g_i \in \mathcal{G}(I)$ and all $g_i$'s may not be distinct.
 Then $g=hc^t{y_{i_1}\cdots   y_{i_{t}}} = g_1 g_2\cdots g_{s-(n+1)-t(n+1)}\\ c^t{y_{i_1}\cdots   y_{i_{t}}}.$
	\vspace*{0.1cm}\\ 
	$\textbf{\underline{\mbox{Case  (2.a)}}}$ Assume that there exist  $y_{i_j}$ and  $y_{i_k}$ for some $j,k \in [t]$ such that $y_{i_j}h y_{i_k}$ can be expressed as the   product of $ {s-(n+1)-t(n+1)} + 1   $ minimal generators of $I.$  
	\vspace*{0.1cm}\\
	Without loss of generality, let  $y_{i_j} = y_{i_1}$ and  $y_{i_k}= y_{i_2}$. Then $y_{i_1}h y_{i_2}\in \mathcal{G}({I}^{s-(n+1)-t(n+1)+1}).$ Here we can express
	\begin{align*}
	g&=hc^t{y_{i_1}\cdots   y_{i_{t}}}\\
	&=(y_{i_1}h y_{i_2})c^t{y_{i_3}\cdots   y_{i_{t}}}\\
	&=(y_{i_1}h y_{i_2})(c)c^{t-1}{y_{i_3}\cdots   y_{i_{t}}}\\
	&=(y_{i_1}h y_{i_2})(x_1\cdots x_{2n}x_{2n+1})c^{t-1}{y_{i_3}\cdots   y_{i_{t}}}\\
	&=(y_{i_1}h y_{i_2})(x_1\cdots x_{2n})c^{t-1}{y_{i_3}\cdots   y_{i_{t}}}x_{2n+1}.
	\end{align*}
	Since $x_1\cdots x_{2n}= (x_1x_2)\cdots (x_{2n-1}x_{2n}) \in \mathcal{G}({I}^{n})$, we have \begin{align*}
	g &= (y_{i_1}h y_{i_2})(x_1\cdots x_{2n})c^{t-1}{y_{i_3}\cdots   y_{i_{t}}}x_{2n+1}\\&
	\in {I}^{s-(n+1)-t(n+1)+1}{I}^{n}(c)^{t-1}\\
	&={I}^{s-(n+1)-(t-1)(n+1)}(c)^{t-1}.
	\end{align*}    
	$\textbf{\underline{\mbox{Case  (2.b)}}}$ Assume that there are no  $y_{i_j}$ and  $y_{i_k}$ for  $j,k \in [t]$ such that $y_{i_j}h y_{i_k}$ can be expressed as the   product of $ {s-(n+1)-t(n+1)} + 1   $ minimal generators of $I.$
	\vspace*{0.1cm}\\  
	Here $y_{i_1} \notin   N[V(C_{2n+1})].$ 	Since $g \in \mathcal{G}({I}^{s-(n+1)})$, there exists some minimal generator $g_j=y_{j_1} y_{j_2}$ of $I$ where $y_{j_1} \in   N_G(y_{i_1})$ such that 
	$y_{i_1}$  pair up  with $y_{j_1}  $ to get some element of $\mathcal{G}(I)$ and $ y_{j_2}$ remains unpaired.\\ 
	If there is $y_{i_{r_1}} \in   N_G(y_{j_2})$ for some $r_1\in [t]$, then $y_{j_2}$  can pair up  with $y_{i_{r_1}}$ and $y_{i_1}g_jy_{i_{r_1}}=y_{i_1}(y_{j_1} y_{j_2})y_{i_{r_1}}=(y_{i_1}y_{j_1}) (y_{j_2}y_{i_{r_1}})  \in \mathcal{G}({I}^{2})$. Thus  $y_{i_1}h y_{i_{r_1}}\in \mathcal{G}({I}^{s-(n+1)-t(n+1)+1}),$ i.e., a contradiction by our assumption.\\
	If  $y_{j_2} \in   N[V(C_{2n+1})],$ then $y_{j_2} $ can combine with $c$ to get some element of $\mathcal{G}(I^{n+1})$ and the proof follows by the same argument as in \textbf{Case (1)}. Otherwise,   there exists some minimal generator  $g_k=y_{j_3} y_{j_4}$ of $I$ where $y_{j_3} \in   N_G(y_{j_2})$  such that $y_{j_2}$  pair up  with
	$y_{j_3}$ to get some element of $\mathcal{G}(I)$  and 
	$ y_{j_4}$ again remains unpaired.\\
	If there is $y_{i_{r_2}} \in   N_G(y_{j_4})$ for some $r_2\in [t]$. Then $y_{j_4}$  can pair up  with $y_{i_{r_2}}$ and $y_{i_1}g_jg_ky_{i_{r_2}}=y_{i_1}(y_{j_1} y_{j_2})(y_{j_3} y_{j_4})y_{i_{r_2}}=(y_{i_1}y_{j_1}) (y_{j_2}y_{j_3})( y_{j_4}y_{i_{r_2}})  \in \mathcal{G}({I}^{3})$. Thus  $y_{i_1}h y_{i_{r_2}}\in \mathcal{G}({I}^{s-(n+1)-t(n+1)+1}),$ i.e., again a contradiction by our assumption.\\
	If  $y_{j_4} \in   N[V(C_{2n+1})],$ then $y_{j_4} $ can combine with $c$  to get some element of $\mathcal{G}(I^{n+1})$ and the proof follows by the same argument as in \textbf{Case (1)}. Otherwise,   there exists some minimal generator  $g_l=y_{j_5} y_{j_6}$ of $I $ where $y_{j_5} \in   N_G(y_{j_4})$  such that $y_{j_4}$  pair up  with
	$y_{j_5}$ to get some element of $\mathcal{G}(I)$ and 
	$ y_{j_6}$ again remains unpaired.\\  If we keep  continue this process, we will find that the unpaired variable can not pair up with any   $y_{i_{r}}$ for some $r\in [t]$. 
	Since $g \in \mathcal{G}({I}^{s-(n+1)})$,
	the unpaired variable must combine with $c$ after some steps. 
	Thus by the same procedure as in Lemma \ref{cont.1},   $ hc{y_{i_1}} $ can be expressed as the product of $s-(n+1)-t(n+1) + n+ 1$ minimal generators of $I,$ i.e., $ hc{y_{i_1}}   \in \mathcal{G}({I}^{s-(n+1)-t(n+1) + n+ 1})=\mathcal{G}({I}^{s-(n+1)-(t-1)(n+1) }).$
	Then  $$  g=hc^t{y_{i_1}\cdots   y_{i_{t}}} = (hc{y_{i_1}} )c^{t-1}{y_{i_2}\cdots   y_{i_{t}}} \in  {I}^{s-(n+1)-(t-1)(n+1)}(c)^{t-1}.$$
	Hence 
	$ \mathcal{G}({I}^{s-(n+1)}) \cap  {I}^{s-(n+1)-t(n+1)}(c)^{t}  \subseteq {I}^{s-(n+1)-(t-1)(n+1)}(c)^{t-1}$ for $2 \leq t \leq k-1 .$\\
	For a fixed $t,$
	\begin{align*}
	&\mathcal{G}({I}^{s-(n+1)}) \cap  {I}^{s-(n+1)-t(n+1)}(c)^{t}  \subseteq {I}^{s-(n+1)-(t-1)(n+1)}(c)^{t-1},\\
	&\mathcal{G}({I}^{s-(n+1)}) \cap  {I}^{s-(n+1)-(t-1)(n+1)}(c)^{t-1}  \subseteq {I}^{s-(n+1)-(t-2)(n+1)}(c)^{t-2},\\
	&\hspace*{2cm}\vdots\hspace*{3cm}\vdots\hspace*{3cm}\vdots\\  
	&\mathcal{G}({I}^{s-(n+1)}) \cap  {I}^{s-(n+1)-2(n+1)}(c)^{2}  \subseteq {I}^{s-(n+1)-(n+1)}(c) = {I}^{s-2(n+1)}(c).
	\end{align*} 
	Therefore for $2 \leq t \leq k-1 ,$
	\begin{align*}
	\mathcal{G}({I}^{s-(n+1)}) \cap  {I}^{s-(n+1)-t(n+1)}(c)^{t}  &\subseteq \mathcal{G}({I}^{s-(n+1)}) \cap {I}^{s-(n+1)-(t-1)(n+1)}(c)^{t-1}\\
	&\subseteq \mathcal{G}({I}^{s-(n+1)}) \cap  {I}^{s-(n+1)-(t-2)(n+1)}(c)^{t-2}\\
	&\hspace*{0.2cm}\vdots\hspace*{2.2cm}\vdots\hspace*{2.2cm}\vdots\\ 
	&\subseteq \mathcal{G}({I}^{s-(n+1)}) \cap  {I}^{s-(n+1)-2(n+1)}(c)^{2}  \\
	&\subseteq   {I}^{s-(n+1)-(n+1)}(c) ={I}^{s-2(n+1)}(c).
	\end{align*}  
	Hence 	$  \mathcal{G}({I}^{s-(n+1)}) \cap  {I}^{s-(n+1)-t(n+1)}(c)^{t} $ $  \subseteq  {I}^{s-2(n+1)}(c)$  for $2 \leq t \leq k-1 .$ 
\end{proof}

\begin{remark}\label{ucc3}
Let $ d_{n+1},d_{n+2},\ldots,d_{2n+1+l} $ be the number of   edge sets of  size $ n+1,n+2,\ldots,2n+1+l ,$ respectively such that each edge set contains a Type-I or Type-II minimum   edge cover for $C_{2n+1}$  in $ G. $ Then the values of $ d_{n+1},d_{n+2},\ldots,d_{2n+1+l} $ are the same as computed in Corollary  \ref{ucc3.}.

\end{remark}

\begin{lemma}\label{uc2}  
	Let $ I = I(G) $ be   the edge ideal of $G$. Let $ s \in   \mathbb{N}  $ and write $ s = k(n + 1) + r  $ for some $ k \in \mathbb{Z}  $ and
$ 0 \leq  r    \leq  n. $ Let $c = x_1\cdots x_{2n+1}.$ Let $ d_{n+1},d_{n+2},\ldots,d_{2n+1+l} $ are same as defined in Remark \ref{ucc3}.   Let $ \mathscr{G}_{s-(n+1)} $ be the set   of elements of $ \mathcal{G}(I^{s-(n+1)}) $ which are not multiple of some element of   $ \mathcal{G}(I^{s-2(n+1)}(c))$  for $ s \geq 2n+2.$  Then for $ s \geq 2n+2,$ 
\vspace*{0.1cm}\\
$\displaystyle{ |\mathscr{G}_{s-(n+1)}| =     \multiset{2n+1+l}{s-(n+1)}  - \bigg[    d_{n+1}\multiset{n+1}{s-2(n+1)}+d_{n+2}\multiset{n+2}{s-2(n+1)-1} +\cdots + }\vspace*{0.2cm}\\
\hspace*{1.9cm}\displaystyle{   d_{2n+1+l} \multiset{2n+1+l}{s-2(n+1)-n-l}        \bigg]}.      
$ 
	
\end{lemma}      
\begin{proof}	
Let $ \nu_{s-(n+1)} $ be the set   of elements of $ \mathcal{G}(I^{s-(n+1)}) $ which are  multiple of some element of   $ \mathcal{G}(I^{s-2(n+1)}(c))$  for $ s \geq 2n+2.$ First we want to  compute $|\nu_{s-(n+1)}| $  for $ s \geq 2n+2.$
\vspace*{0.1cm}\\
Let $q$ be an element of $ \mathcal{G}(I^{s-(n+1)}) $ which is  multiple of some element of   $ \mathcal{G}(I^{s-2(n+1)}(c)) $, i.e.,
 $q \in \mathcal{G}(I^{s-(n+1)}) \cap  I^{s-2(n+1)}(c) .$ By Lemma \ref{cont.1}, $q \in I^{s-2(n+1)} (I^{n+1}\cap (c)).$   
 Then $q = q_1 q_2$ for some $q_1 \in \mathcal{G}(I^{s-2(n+1)})$ and $q_2 \in  \mathcal{G}(I^{n+1}) \cap (c).$ Note that   $q_2 $ is  an element of   $\mathcal{G}(I^{n+1})$ and a multiple of $c.$ Since the product of any $n$ minimal generators of $I$ can not be a multiple of $c$  and   $\gamma^\prime(C_{2n+1}) = \gamma^{\prime\prime}(C_{2n+1}) =  n+1,$        
$q_2$ is the product of $n+1$ minimal generators of $I$ corresponding to some Type-I or Type-II minimum edge cover for $C_{2n+1}$ in $G.$ We can say  $q$ is the product of $s-(n+1)$ minimal generators of $I$ among which  $n+1$ minimal generators  are always corresponding to some Type-I or Type-II minimum edge cover for $C_{2n+1}$ in $G.$
\vspace*{0.1cm}\\
Fix $2n+2 \leq s \leq 3n+2+l.$ Then $n+1 \leq s-(n+1) \leq 2n+1+l.$  We want to choose $s-(n+1)$ minimal generators of $I$
containing $n+1$ minimal generators  corresponding to some Type-I or Type-II minimum edge cover for $C_{2n+1}$ in $G$
with repetition. 
While choosing  such $s-(n+1)$ minimal generators of $I$,  we can use only the distinct minimal generators of $I$ corresponding to some  edge set of size $n+1+k$   containing a Type-I or Type-II minimum edge cover of $C_{2n+1}$ in $G$ for each $k$ where $0 \leq k \leq s-2(n+1).$  So  $q$ is the product of  $s-(n+1)$ minimal generators of $I$ among which $n+1+k$ minimal generators  are corresponding to some edge set of size $n+1+k$ containing a Type-I or Type-II minimum edge cover for $C_{2n+1}$ in $G$ and the remaining from those $n+1+k$ minimal generators of $I$ with repetition  where $0 \leq k  \leq s-2(n+1).$ Thus
\vspace*{0.2cm}\\ 
 $\displaystyle{ |\nu_{n+1}| = |\nu_{(2n+2)-(n+1)}| =       d_{n+1}\multiset{n+1}{0}       }    
, $ \\

$\displaystyle{|\nu_{n+2}| = |\nu_{(2n+3)-(n+1)}| =      }\displaystyle{     d_{n+1}\multiset{n+1}{1}+d_{n+2}}  \displaystyle{\multiset{n+2}{0}           }   
, $  
\vspace*{0.2cm}\\
 \hspace*{0.5cm}\vdots\hspace*{2cm}\vdots \hspace*{2.5cm} \vdots 
\vspace*{0.2cm}\\
$\displaystyle{ |\nu_{2n+1+l}| = |\nu_{(3n+2+l)-(n+1)}| =          d_{n+1}\multiset{n+1}{n+l}+d_{n+2}}\displaystyle{\multiset{n+2}{n-1+l}}\displaystyle{+\cdots+ d_{2n+1+l} \multiset{2n+1+l}{0}        }.$
\vspace*{0.4cm}\\
Therefore by Remark \ref{unique3} for $2n+2 \leq s \leq 3n+2+l,$ we have  
$$\displaystyle{ |\mathscr{G}_{s-(n+1)}| =   \multiset{2n+1+l}{ s-(n+1)}  -   |\nu_{s-(n+1)}|      }  
.$$

Fix $ s \geq 3n+3+l.$ Then $ s-(n+1) \geq 2n+2+l.$ We want to choose $s-(n+1)$ minimal generators of $I$
containing $n+1$ minimal generators  corresponding to some Type-I or Type-II minimum edge cover for $C_{2n+1}$ in $G$
with repetition. 
While choosing  such $s-(n+1)$ minimal generators of $I$,
we can use only the distinct minimal generators of $I$ corresponding to some edge set of size $n+1+k$   containing a Type-I or Type-II minimum edge cover of $C_{2n+1}$ in $G$ for each $k$ where $0 \leq k \leq n+l.$  So  $q$ is the product of  $s-(n+1)$ minimal generators of $I$ among which $n+1+k$ minimal generators  are      corresponding to some edge set of size $n+1+k$ containing a Type-I or Type-II minimum edge cover for $C_{2n+1}$ in $G$ and the remaining from those $n+1+k$ minimal generators of $I$ with repetition  where $0 \leq k \leq n+l.$
\vspace*{0.2cm}\\
Thus for $ s \geq 3n+3+l ,$   
$\displaystyle{ |\nu_{s-(n+1)}| =       d_{n+1}\multiset{n+1}{s-2(n+1)}+d_{n+2}\multiset{n+2}{s-2(n+1)-1}}
\vspace*{0.2cm}\\
\hspace*{5.8cm}\displaystyle{+\cdots+  d_{2n+1+l} \multiset{2n+1+l}{s-2(n+1)-n-l}        }    
.$
\vspace*{0.2cm}\\
Therefore by Remark \ref{unique3} for $ s \geq 3n+3+l ,$ we have   
$$\displaystyle{ |\mathscr{G}_{s-(n+1)}| =   \multiset{2n+1+l}{ s-(n+1)}  -   |\nu_{s-(n+1)}|      }    
.$$
Hence for $ s \geq 2n+2,$
\begin{align*}                 
 |\mathscr{G}_{s-(n+1)}| &=\displaystyle{    \multiset{2n+1+l}{s-(n+1)}  - |\nu_{s-(n+1)}|  }\\
  &= \displaystyle{    \multiset{2n+1+l}{s-(n+1)}  - \bigg[    d_{n+1}\multiset{n+1}{s-2(n+1)}+d_{n+2}\multiset{n+2}{s-2(n+1)-1} +\cdots + }\\  
&\hspace*{0.5cm}\displaystyle{   d_{2n+1+l} \multiset{2n+1+l}{s-2(n+1)-n-l}        \bigg]}.
\end{align*}
\end{proof}  

  Let $ |\mathscr{G}_{s-(n+1)}|  $ is same as computed in Lemma \ref{uc2}   for $ s \geq 2n+2.$  For $ n + 1 \leq   s \leq 2n + 1 ,$ define $\displaystyle{ |\mathscr{G}_{s-(n+1)}| =  \multiset{2n+1+l}{s-(n+1)}   }$. In the next theorem, we compute all the symbolic defects of the edge ideal of $G.$    	
\begin{theorem}\label{sd.unicyclic}
	Let $ I = I(G) $ be   the edge ideal of $G$. Let $ s \in   \mathbb{N}  $ and write $ s = k(n + 1) + r  $ for some $ k \in \mathbb{Z}  $ and $ 0 \leq  r    \leq  n. $  Let $ d_{n+1},d_{n+2},\ldots,d_{2n+1+l} $ are same as defined in Remark \ref{ucc3}.   Then
	\begin{enumerate}
		\item 	for $ n + 1 \leq   s \leq 2n + 1   ,$
		 $$\displaystyle{ \sdefect({{ I}}, s) =|\mathscr{G}_{s-(n+1)}|  =    \multiset{2n+1+l}{s-(n+1)} },$$
 	\item for   $ s \geq 2n+2, $ 	$$\displaystyle{ \sdefect({{ I}}, s) =    |\mathscr{G}_{s-(n+1)}|      }    
 	+   \sdefect({{ I}}, s-(n+1))  .$$    
	\end{enumerate}

\end{theorem}
\begin{proof}    
 By Lemma \ref{unicyclic} for $ s = k(n + 1) + r  $ where $ k \in \mathbb{Z}  $ and $ 0 \leq  r    \leq  n, $      we have  $${I}^{(s)} =   \displaystyle{ \sum_{t=0}^{k}  {I}^{s-t(n+1)}  (c)^t  }~ \mbox{where}~ c= x_1 \cdots x_{2n+1} .$$   
Fix	 $ n + 1 \leq   s \leq 2n + 1 .$ Then  
$$ {I}^{(s)} = I^{s}  +   \sdefectideal (I^{(s)}) = \displaystyle{ I^{s}  +   {I}^{s-(n+1)}  (c)  }.$$

Hence by Remark \ref{unique3} for $ n + 1 \leq   s \leq 2n + 1 ,  $ we have 
$$\sdefect({{ I}}, s) 
 = |\mathcal{G}(\sdefectideal (I^{(s)}) )|=  |\mathcal{G}({I}^{s-(n+1)}  (c)) |
  =   |\mathcal{G}({I}^{s-(n+1)}  ) |
 =\displaystyle{  \multiset{2n+1+l}{s-(n+1)} }.$$
  Fix   $  s \geq 2n + 2 $ and write $ s = k(n + 1) + r  $ for some $ k \in \mathbb{Z}  $ and
$ 0 \leq  r    \leq  n. $ Then
\begin{align}
  {I}^{(s)}
  &= \displaystyle{ I^{s}  +    \sdefectideal (I^{(s)})  }\label{4}\\
  &=   \displaystyle{ I^s +  \sum_{t=1}^{k}  {I}^{s-t(n+1)}  (c)^t  }\nonumber\\
  &=   \displaystyle{ I^s +  (c)\sum_{t=0}^{k-1}  {I}^{s-(n+1)-t(n+1)}  (c)^t  }\nonumber   \\ 
&=  \displaystyle{ I^s +  (c) [ {I}^{s-(n+1)} + {I}^{s-  2(n+1)}(c) + {I}^{s-  3(n+1)}(c)^2 + \cdots  + {I}^{s-  k(n+1)}(c)^{k-1} ]  }\label{5}\\
  &= \displaystyle{ I^s +  (c) [ {I}^{s-(n+1)} +  \sum_{t=1}^{k-1}  {I}^{s-(n+1)-t(n+1)}   (c)^t]  }\label{6}\\
  &= \displaystyle{ I^{s}  + (c) [ {I}^{s-(n+1)}   + \sdefectideal({I}^{({s-(n+1)})}) ]  }\label{7}.     	
\end{align}        
Thus by (\ref{4}), (\ref{5}), (\ref{6}) and (\ref{7}) for $ s \geq 2n + 2, $
\begin{align*}  
\sdefect(I,s)    & = |\mathcal{G}(\sdefectideal (I^{(s)}))|\\ 
&=  |\mathcal{G}((c) [ {I}^{s-(n+1)}   + \sdefectideal({I}^{({s-(n+1)})}) ])|\\
&=  |\mathcal{G}(  {I}^{s-(n+1)}   + \sdefectideal({I}^{({s-(n+1)})}) )|\\
&=  |\mathcal{G}(  {I}^{s-(n+1)}   + \displaystyle{\sum_{t=1}^{k-1}  {I}^{s-(n+1)-t(n+1)}   (c)^t} )|,
\end{align*}

i.e.,  $ \sdefect(I,s) $  is the  number of elements of $  \mathcal{G}(I^{s-(n+1)} ) $ which are not multiple of some element
  of  $\mathcal{G}(\displaystyle{\sum_{t=1}^{k-1}  {I}^{s-(n+1)-t(n+1)}   (c)^t} ) +  |\mathcal{G}(\displaystyle{\sum_{t=1}^{k-1}  {I}^{s-(n+1)-t(n+1)}   (c)^t} )|. $    By  Lemma \ref{uc1}, any element of $\mathcal{G}({I}^{s-(n+1)}) $ which   lies in the ideal $  \displaystyle{\sum_{t=1}^{k-1}  {I}^{s-(n+1)-t(n+1)}   (c)^t}$ must be a multiple of some element of $\mathcal{G}({I}^{s-2(n+1)}(c)) .$  This implies that  $ \sdefect(I,s) $ is  the number of elements of  $ \mathcal{G}(I^{s-(n+1)} ) $  which are not multiple of some element   of  $ \mathcal{G}(I^{s-2(n+1)}(c)) +  |\mathcal{G}(\displaystyle{\sum_{t=1}^{k-1}  {I}^{s-(n+1)-t(n+1)}   (c)^t} )|  
= |\mathscr{G}_{s-(n+1)}|  +  \sdefect(I,s-(n+1)). $
\end{proof}  	
 Next, we give  the exact values of symbolic defects of the edge ideal of $G$  in terms of  $|\mathscr{G}_{j}|$'s in the following corollary.   
\begin{corollary}\label{sd.unicyclic1}
	Let $ I = I(G) $ be   the edge ideal of $G$. Let $ s \in   \mathbb{N}  $ and write $ s = k(n + 1) + r  $ for some $ k \in \mathbb{Z}  $ and $ 0 \leq  r    \leq  n. $ 	Then
	
\begin{equation*}  
\sdefect({{ I}}, s)= 
\begin{cases}
&|\mathscr{G}_{(k-1)(n+1)}|+|\mathscr{G}_{(k-2)(n+1)}|+\cdots+|\mathscr{G}_{(n+1)}|   +  |\mathscr{G}_{0}| 
\vspace*{0.2cm}~\mbox{if}~  s = k(n + 1) \vspace*{0.2cm}\\
&|\mathscr{G}_{(k-1)(n+1)+1}|+|\mathscr{G}_{(k-2)(n+1)+1}|+\cdots+|\mathscr{G}_{(n+1)+1}|   +   |\mathscr{G}_{1}| 
\vspace*{0.2cm}\\ &\hspace*{7.8cm}~\mbox{if}~  s = k(n + 1) + 1\vspace*{0.2cm}\\
&  \hspace*{1cm}   \vdots\hspace*{3cm}  \vdots\hspace*{3cm}  \vdots\vspace*{0.2cm}\\
&|\mathscr{G}_{(k-1)(n+1)+n}|+|\mathscr{G}_{(k-2)(n+1)+n}|+\cdots+|\mathscr{G}_{(n+1)+n}|   +   |\mathscr{G}_{n}|
\vspace*{0.2cm}\\ 
&\hspace*{7.8cm}~\mbox{if}~  s = k(n + 1) + n       
\end{cases}.  
\end{equation*}	
	
\end{corollary}
\begin{proof}
By Theorem \ref{sd.unicyclic}	
for   $ s \geq n+1, $ we have 	$$\displaystyle{ \sdefect({{ I}}, s) =    |\mathscr{G}_{s-(n+1)}|      }    
	+   \sdefect({{ I}}, s-(n+1))  .$$
	If $ s = k(n+1), $ then 
	\begin{align*}
 \sdefect({{ I}}, s) &=    |\mathscr{G}_{(k-1)(n+1)}|          
+   \sdefect({{ I}}, (k-1)(n+1))\\
&=    |\mathscr{G}_{(k-1)(n+1)}|          
+   |\mathscr{G}_{(k-2)(n+1)}|          
+   \sdefect({{ I}}, (k-2)(n+1))\\
&  \hspace*{1cm}   \vdots\hspace*{3cm}  \vdots\hspace*{3cm}  \vdots\vspace*{0.2cm}\\\
&=|\mathscr{G}_{(k-1)(n+1)}|+|\mathscr{G}_{(k-2)(n+1)}|+\cdots+|\mathscr{G}_{(n+1)}|   +   \sdefect({{ I}}, (n+1))\\
&=|\mathscr{G}_{(k-1)(n+1)}|+|\mathscr{G}_{(k-2)(n+1)}|+\cdots+|\mathscr{G}_{(n+1)}|   +  |\mathscr{G}_{0}|. 
	\end{align*}
	Similarly if   $ s = k(n+1)+1, $ then 
	$$\sdefect({{ I}}, s) =|\mathscr{G}_{(k-1)(n+1)+1}|+|\mathscr{G}_{(k-2)(n+1)+1}|+\cdots+|\mathscr{G}_{(n+1)+1}|   +   |\mathscr{G}_{1}| .$$
	 \hspace*{3cm}   \vdots\hspace*{3cm}  \vdots\hspace*{3cm}  \vdots\vspace*{0.2cm}\\
 If   $ s = k(n+1)+n, $ then 
$$\sdefect({{ I}}, s) =|\mathscr{G}_{(k-1)(n+1)+n}|+|\mathscr{G}_{(k-2)(n+1)+n}|+\cdots+|\mathscr{G}_{(n+1)+n}|   +   |\mathscr{G}_{n}| .$$	 
\end{proof}

 In the following remark, we give one procedure to find the quasi-polynomial associated with the symbolic defects of edge ideal of $G$.

\begin{remark}\label{sd.unicyclic.2}
	Let $ I = I(G) $ be   the edge ideal of $G$. Let $ s \in   \mathbb{N}  $ and write $ s = k(n + 1) + r  $ for some $ k \in \mathbb{Z}  $ and $ 0 \leq  r    \leq  n. $   By Lemma  \ref{quasi}, the $\sdefect({{ I}}, s)$
	is a quasi-polynomial of quasi-period $n+1$.
\begin{align*}
\mbox{Let}~ f_0(k) &=|\mathscr{G}_{(k-1)(n+1)}|+|\mathscr{G}_{(k-2)(n+1)}|+\cdots+|\mathscr{G}_{(n+1)}|   +  |\mathscr{G}_{0}| 
\vspace*{0.2cm}~\mbox{if}~  s = k(n + 1) ,\vspace*{0.2cm}\\
f_1(k)&=|\mathscr{G}_{(k-1)(n+1)+1}|+|\mathscr{G}_{(k-2)(n+1)+1}|+\cdots+|\mathscr{G}_{(n+1)+1}|   +   |\mathscr{G}_{1}| 
~\mbox{if}~  s = k(n + 1) + 1,\vspace*{0.2cm}\\
&  \hspace*{1cm}   \vdots\hspace*{3cm}  \vdots\hspace*{3cm}  \vdots\vspace*{0.2cm}\\
f_{n}(k)&=|\mathscr{G}_{(k-1)(n+1)+n}|+|\mathscr{G}_{(k-2)(n+1)+n}|+\cdots+|\mathscr{G}_{(n+1)+n}|   +   |\mathscr{G}_{n}|
~\mbox{if}~  s = k(n + 1) + n. 
\end{align*}                  
Then by Corollary \ref{sd.unicyclic1}, we can say

\begin{equation*}  
\sdefect({{ I}}, s)= 
\begin{cases}
&f_0(k) 
\vspace*{0.2cm}~\mbox{if}~  s = k(n + 1) \\
&f_1(k)   ~\mbox{if}~  s = k(n + 1) + 1\vspace*{0.2cm}\\
&  \hspace*{0.5cm}   \vdots\hspace*{2cm}  \vdots\vspace*{0.2cm}\\
&f_{n}(k)  ~\mbox{if}~  s = k(n + 1) + n       
\end{cases}
\end{equation*}	
 grows as a  quasi-polynomial in $ s $ with quasi-period $n+1.
 $ 
\end{remark}    
We have seen that it is very difficult to compute explicitly the quasi-polynomial associated with the symbolic defects  of  edge ideal of any unicyclic graph on $n$ vertices where $n$ is large. We  compute explicitly the quasi-polynomial associated with the symbolic defects  of  edge ideal of an odd cycle of length $5$ in the following example.
           
\begin{example}    
Consider $G$ to be the odd cycle $C_5=(x_1,\ldots,x_5).$ Let $I = I(C_5)$ be the edge ideal of $C_5$. 
	Let $ d_{3},d_{4},d_{5} $ be the number of   edge sets of  size $ 3,4,5,$ respectively such that each edge set contains a Type-I or Type-II minimum   edge cover for $C_{5}$. Note that there is no  Type-II  minimum   edge cover for $C_{5}$. Then by Lemma \ref{type1}, we have  $ d_{3}=5,$ $d_{4}=5$ and $d_{5} = 1.$ Let $ |\mathscr{G}_{s-3}|  $ is same as defined earlier   for $ s \geq 3.$ So
	\vspace*{0.1cm}\\ $\displaystyle{ |\mathscr{G}_{0}| =    \multiset{5}{0} },$  $\displaystyle{ |\mathscr{G}_{1}| =    \multiset{5}{1} ,}$  $\displaystyle{ |\mathscr{G}_{2}| =    \multiset{5}{2} }$ and for $ s \geq 6,$ 
\vspace*{0.2cm}\\
$\displaystyle{ |\mathscr{G}_{s-3}| =     \multiset{5}{s-3}  - \bigg[    5\multiset{3}{(s-3)-3}+5\multiset{4}{(s-3)-4} +}\displaystyle{   1 \multiset{5}{(s-3)-5}        \bigg]}    $.
\vspace*{0.1cm}\\
	Fix $ s = 3k.$ Then by Corollary \ref{sd.unicyclic1}, we have
		$$ \sdefect(I,s)  = |\mathscr{G}_{3(k-1)}|+|\mathscr{G}_{3(k-2)}|+\cdots+|\mathscr{G}_{3(2)}| +|\mathscr{G}_{3(1)}|   +  |\mathscr{G}_{0}| .$$	
	Here  $\displaystyle{|\mathscr{G}_{0}| = \binom{4}{0}=1}$ and   $\displaystyle{ |\mathscr{G}_{3}| =     \multiset{5}{3}  -    5\multiset{3}{0}   =   \binom{7}{3}  - 5\binom{2}{0}=30  }$.
\vspace*{0.1cm}\\  
Now we assume that $s-3 \geq 6.$ Then  
\begin{align*}       
|\mathscr{G}_{3(k-1)}| &= \displaystyle{  \multiset{5}{3(k-1)}}   - \bigg[    5\multiset{3}{3(k-1)-3}+5\multiset{4}{3(k-1)-4} +  1 \multiset{5}{3(k-1)-5}        \bigg]\\
&=  \binom{3k+1}{3k-3}   - \bigg[    5 \binom{3k-4}{3k-6}+5 \binom{3k-4}{3k-7} +  1  \binom{3k-4}{3k-8}       \bigg]\\
&=\frac{45}{2}k^2 - \frac{75}{2}k + 15.        
\end{align*} 
Similarly, $\displaystyle{|\mathscr{G}_{3(k-i)}|=\frac{45}{2}(k-i+1)^2 - \frac{75}{2}(k-i+1) + 15  }$ for $2 \leq i \leq k-2$. Hence   
\begin{align*}    
\sdefect(I,s)  &= |\mathscr{G}_{3(k-1)}|+|\mathscr{G}_{3(k-2)}|+\cdots+|\mathscr{G}_{3(2)}| +|\mathscr{G}_{3(1)}|   +  |\mathscr{G}_{0}|\\
&=\big[|\mathscr{G}_{3(k-1)}|+|\mathscr{G}_{3(k-2)}|+\cdots+|\mathscr{G}_{3(2)}|\big] +|\mathscr{G}_{3(1)}|   +  |\mathscr{G}_{0}|\\
&=\bigg[\frac{45}{2}\big[k^2+(k-1)^2+\cdots+3^2\big] - \frac{75}{2}\big[k+(k-1)+\cdots+3\big] + 15(k-2)\bigg] + 30 + 1\\
&=\bigg[\frac{45}{2}\big[k^2+(k-1)^2+\cdots+1^2\big] - \frac{75}{2}\big[k+(k-1)+\cdots+1\big] + 15k\bigg]\\
&\hspace*{0.4cm}- \frac{45}{2}\big[2^2+1^2\big] + \frac{75}{2}[2+1] - 15(2) + 30 + 1\\
&=\bigg[\frac{45}{2}\frac{k(k+1)(2k+1)}{6} - \frac{75}{2}\frac{k(k+1)}{2} + 15k\bigg] + 1\\
&=\frac{15}{2}k^3 - \frac{15}{2}k^2 + 1
\end{align*}    
 
Fix $ s = 3k+1.$  Then by Corollary \ref{sd.unicyclic1}, we have	
$$ \sdefect(I,s)  = |\mathscr{G}_{3(k-1)+1}|+|\mathscr{G}_{3(k-2)+1}|+\cdots+|\mathscr{G}_{3(2)+1}| +|\mathscr{G}_{3(1)+1}|   +  |\mathscr{G}_{1}| .$$    

Here  $\displaystyle{|\mathscr{G}_{1}| = \binom{5}{1}=5}$ and   $\displaystyle{ |\mathscr{G}_{4}| =     \multiset{5}{4}  -    5\multiset{3}{1}   -  5\multiset{4}{0}  =   \binom{8}{4}  - 5\binom{3}{1} - 5\binom{3}{0}  =50  }$.
\vspace*{0.1cm}\\                
Now we assume that  $s-3 \geq 7.$ Then  
\begin{align*}       
|\mathscr{G}_{3(k-1)+1}| &= \displaystyle{  \multiset{5}{3(k-1)+1}}   - \bigg[    5\multiset{3}{3(k-1)-2}+5\multiset{4}{3(k-1)-3} +  1 \multiset{5}{3(k-1)-4}        \bigg]\\
&=  \binom{3k+2}{3k-2}   - \bigg[    5 \binom{3k-3}{3k-5}+5 \binom{3k-3}{3k-6} +  1  \binom{3k-3}{3k-7}       \bigg]\\
&=\frac{45}{2}k^2 - \frac{45}{2}k +  5.        
\end{align*} 
Similarly, $\displaystyle{|\mathscr{G}_{3(k-i)+1}|=\frac{45}{2}(k-i+1)^2 - \frac{45}{2}(k-i+1) + 5  }$ for $2 \leq i \leq k-2$. Hence   
\begin{align*}    
\sdefect(I,s)  &= |\mathscr{G}_{3(k-1)+  1}|+|\mathscr{G}_{3(k-2)+  1}|+\cdots+|\mathscr{G}_{3(2)+  1}| +|\mathscr{G}_{3(1)+  1}|   +  |\mathscr{G}_{1}|\\
&=\big[|\mathscr{G}_{3(k-1)+  1}|+|\mathscr{G}_{3(k-2)+  1}|+\cdots+|\mathscr{G}_{3(2)+  1}|\big] +|\mathscr{G}_{3(1)+  1}|   +  |\mathscr{G}_{1}|\\
&=\bigg[\frac{45}{2}\big[k^2+(k-1)^2+\cdots+3^2\big] - \frac{45}{2}\big[k+(k-1)+\cdots+3\big] + 5(k-2)\bigg] + 50 + 5\\
&=\bigg[\frac{45}{2}\big[k^2+(k-1)^2+\cdots+1^2\big] - \frac{45}{2}\big[k+(k-1)+\cdots+1\big] + 5k\bigg]\\
&\hspace*{0.4cm}- \frac{45}{2}\big[2^2+1^2\big] + \frac{45}{2}[2+1] -   5(2) + 55\\
&=\frac{45}{2}\frac{k(k+1)(2k+1)}{6} - \frac{45}{2}\frac{k(k+1)}{2} + 5k\\
&=\frac{15}{2}k^3 - \frac{5}{2}k.  
\end{align*}

Fix $ s = 3k+2.$  Then by Corollary   \ref{sd.unicyclic1}, we have
    	$$ \sdefect(I,s)  = |\mathscr{G}_{3(k-1)+2}|+|\mathscr{G}_{3(k-2)+2}|+\cdots+|\mathscr{G}_{3(2)+2}| +|\mathscr{G}_{3(1)+2}|   +  |\mathscr{G}_{2}| .$$    
Here  $\displaystyle{|\mathscr{G}_{2}| = \binom{6}{2}=15}$.     
Now we assume that $s-3 \geq 5.$  Then  
\begin{align*}           
|\mathscr{G}_{3(k-1)+2}| &= \displaystyle{  \multiset{5}{3(k-1)+2}}   - \bigg[    5\multiset{3}{3(k-1)-1}+5\multiset{4}{3(k-1)-2} +  1 \multiset{5}{3(k-1)-3}        \bigg]\\
&=  \binom{3k+3}{3k-1}   - \bigg[    5 \binom{3k-2}{3k-4}+5 \binom{3k-2}{3k-5} +  1  \binom{3k-2}{3k-6}       \bigg]\\
&=\frac{45}{2}k^2 - \frac{15}{2}k.        
\end{align*} 
Similarly, $\displaystyle{|\mathscr{G}_{3(k-i)+2}|=\frac{45}{2}(k-i+1)^2 - \frac{15}{2}(k-i+1)   }$ for $2 \leq i \leq k-1$. Hence  
\begin{align*}    
\sdefect(I,s)  &= |\mathscr{G}_{3(k-1)+  2}|+|\mathscr{G}_{3(k-2)+  2}|+\cdots+|\mathscr{G}_{3(2)+  2}| +|\mathscr{G}_{3(1)+  2}|   +  |\mathscr{G}_{2}|\\
&=\big[|\mathscr{G}_{3(k-1)+  2}|+|\mathscr{G}_{3(k-2)+  2}|+\cdots+|\mathscr{G}_{3(2)+  2}| +|\mathscr{G}_{3(1)+  2}|\big]   +  |\mathscr{G}_{2}|\\
&=\bigg[\frac{45}{2}\big[k^2+(k-1)^2+\cdots+2^2\big] - \frac{15}{2}\big[k+(k-1)+\cdots+2\big] \bigg]    + 15\\
&=\bigg[\frac{45}{2}\big[k^2+(k-1)^2+\cdots+1^2\big] - \frac{15}{2}\big[k+(k-1)+\cdots+1\big]  \bigg]\\
&\hspace*{0.4cm}- \frac{45}{2}(1^2) + \frac{15}{2}(1)  + 15\\
&=\frac{45}{2}\frac{k(k+1)(2k+1)}{6} - \frac{15}{2}\frac{k(k+1)}{2}  \\
&=\frac{15}{2}k^3 + \frac{15}{2}k^2.  
\end{align*}
Thus by Remark \ref{sd.unicyclic.2}, we can say

\begin{equation*}  
\sdefect({{ I}}, s)= 
\begin{cases}
&\displaystyle{\frac{15}{2}k^3 - \frac{15}{2}k^2 + 1}
\vspace*{0.2cm}~\mbox{if}~  s = 3k \\
&\displaystyle{\frac{15}{2}k^3 - \frac{5}{2}k}   ~\mbox{if}~  s = 3k + 1\vspace*{0.2cm}\\
&\displaystyle{\frac{15}{2}k^3 + \frac{15}{2}k^2}  ~\mbox{if}~  s = 3k + 2       
\end{cases}
\end{equation*}	
grows as a  quasi-polynomial in $ s $ with quasi-period $3.$
\end{example}

Next we compute explicitly the quasi-polynomial associated with the symbolic defects  of  edge ideal of a small  unicyclic graph.

\begin{example}    
	Consider $G$ to be the  unicyclic graph  in Figure \ref{fig.2}. Let $I = I(G)$ be the edge ideal of $G$. 
	Let $ d_{2},d_{3},d_{4} $ be the number of   edge sets of  size $ 2,3,4,$ respectively such that each edge set contains a Type-I or Type-II minimum   edge cover for $C_{3}$.  
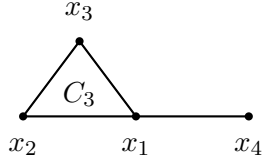
\begin{figure}[!ht]
	\begin{tikzpicture}[scale=1]
	\begin{scope}[ thick, every node/.style={sloped,allow upside down}] 
	\definecolor{ultramarine}{rgb}{0.07, 0.04, 0.56} 
	\definecolor{zaffre}{rgb}{0.0, 0.08, 0.66}
	
\draw[fill][-,thick] (0,0) --(1.5,0);  
	\draw[fill][-,thick](0,0) --(0.75,1);
\draw[fill][-,thick](0.75,1) --(1.5,0);
\draw[fill][-,thick]  (1.5,0) --(3,0);

	\draw[fill][-,thick] (0,0) circle [radius=0.04];
	\draw[fill][-,thick] (1.5,0) circle [radius=0.04];
	\draw[fill][-,thick] (0.75,1) circle [radius=0.04];
	\draw[fill][-,thick] (3,0) circle [radius=0.04];

	\node at (0,-0.4) {$x_2$};
	\node at (1.5,-0.4) {$x_1$};
	\node at (0.75,1.4) {$x_{3}$};  
	\node at (3,-0.4) {$x_{4}$};

	\node at (0.75,0.3) {$C_3$};	

	\end{scope}
	\end{tikzpicture}
	\caption{A unicyclic graph $G$ with an odd  cycle  $C_3$ of length $3$.}\label{fig.2}
\end{figure}  	
Then by Lemma \ref{type1}, we have  $ d_{2}=4,$ $d_{3}=4$ and $d_{4} = 1.$ Let $ |\mathscr{G}_{s-2}|  $ is same as defined earlier   for $ s \geq 2.$ So
 $\displaystyle{ |\mathscr{G}_{0}| =    \multiset{4}{0} },$  $\displaystyle{ |\mathscr{G}_{1}| =    \multiset{4}{1}  }$   and for $ s \geq 4,$ 
	\vspace*{0.2cm}\\
	$\displaystyle{ |\mathscr{G}_{s-2}| =     \multiset{4}{s-2}  - \bigg[    4\multiset{2}{(s-2)-2}+4\multiset{3}{(s-2)-3} +}\displaystyle{   1 \multiset{4}{(s-2)-4}        \bigg]}  $.
	\vspace*{0.1cm}\\
	Fix $ s = 2k.$ Then by Corollary \ref{sd.unicyclic1}, we have
	$$ \sdefect(I,s)  = |\mathscr{G}_{2(k-1)}|+|\mathscr{G}_{2(k-2)}|+\cdots+|\mathscr{G}_{2(2)}| +|\mathscr{G}_{2(1)}|   +  |\mathscr{G}_{0}| .$$	
	Here  $\displaystyle{|\mathscr{G}_{0}| = \binom{3}{0}=1}$ and   $\displaystyle{ |\mathscr{G}_{2}| =     \multiset{4}{2}  -    4\multiset{2}{0}   =   \binom{5}{2}  - 4\binom{1}{0}=6  }$.
	\vspace*{0.1cm}\\  
	Now we assume that $s-2 \geq 4.$ Then  
	\begin{align*}       
	|\mathscr{G}_{2(k-1)}| &= \displaystyle{  \multiset{4}{2(k-1)}}   - \bigg[    4\multiset{2}{2(k-1)-2}+4\multiset{3}{2(k-1)-3} +  1 \multiset{4}{2(k-1)-4}        \bigg]\\
	&=  \binom{2k+1}{2k-2}   - \bigg[    4 \binom{2k-3}{2k-4}+4 \binom{2k-3}{2k-5} +  1  \binom{2k-3}{2k-6}       \bigg]\\
	&=4k - 2.        
	\end{align*} 
	Similarly, $\displaystyle{|\mathscr{G}_{2(k-i)}|=4(k-i+1) - 2  }$ for $2 \leq i \leq k-2$. Hence   
	\begin{align*}    
	\sdefect(I,s)  &= |\mathscr{G}_{2(k-1)}|+|\mathscr{G}_{2(k-2)}|+\cdots+|\mathscr{G}_{2(2)}| +|\mathscr{G}_{2(1)}|   +  |\mathscr{G}_{0}|\\
	&=\big[|\mathscr{G}_{2(k-1)}|+|\mathscr{G}_{2(k-2)}|+\cdots+|\mathscr{G}_{2(2)}|\big] +|\mathscr{G}_{2(1)}|   +  |\mathscr{G}_{0}|\\
	&=\big[ 4\big[k+(k-1)+\cdots+3\big] - 2(k-2)\big] + 6 + 1\\
	&= 4\big[k+(k-1)+\cdots+1\big]  -   4[2+1]  - 2k  + 4   +  6 + 1\\
	&= 2k(k+1)   - 2k  - 1\\
	&=  {2}k^2 - 1
	\end{align*}    
	
	Fix $ s = 2k+1.$  Then by Corollary \ref{sd.unicyclic1}, we have	
	$$ \sdefect(I,s)  = |\mathscr{G}_{2(k-1)+1}|+|\mathscr{G}_{2(k-2)+1}|+\cdots+|\mathscr{G}_{2(2)+1}| +|\mathscr{G}_{2(1)+1}|   +  |\mathscr{G}_{1}| .$$    
	
	Here  $\displaystyle{|\mathscr{G}_{1}| = \binom{4}{1}=4}$ and   $\displaystyle{ |\mathscr{G}_{3}| =     \multiset{4}{3}  -    4\multiset{2}{1}   -  4\multiset{3}{0}  =   \binom{6}{3}  - 4\binom{2}{1} - 4\binom{2}{0}  =8  }$.
	\vspace*{0.1cm}\\                
	Now we assume that  $s-2 \geq  5  .$ Then    
	\begin{align*}       
	|\mathscr{G}_{2(k-1)+1}| &= \displaystyle{  \multiset{4}{2(k-1)+1}}   - \bigg[    4\multiset{2}{2(k-1)-1}+4\multiset{3}{2(k-1)-2} +  1 \multiset{4}{2(k-1)-3}        \bigg]\\
	&=  \binom{2k+2}{2k-1}   - \bigg[    4 \binom{2k-2}{2k-3}+4 \binom{2k-2}{2k-4} +  1  \binom{2k-2}{2k-5}       \bigg]\\
	&=4k.        
	\end{align*} 
	Similarly, $\displaystyle{|\mathscr{G}_{2(k-i)+1}|=  4(k-i+1)   }$ for $2 \leq i \leq k-2$. Hence      
	\begin{align*}    
	\sdefect(I,s)  &= |\mathscr{G}_{2(k-1)+  1}|+|\mathscr{G}_{2(k-2)+  1}|+\cdots+|\mathscr{G}_{2(2)+  1}| +|\mathscr{G}_{2(1)+  1}|   +  |\mathscr{G}_{1}|\\
	&=  4k + 4(k-1) + \cdots + 4(3) + 8 + 4\\
	&=  4[k + (k-1) + \cdots  + 3 +  2 + 1]\\
	&=2k^2 + 2k.  
	\end{align*}

	Thus by Remark \ref{sd.unicyclic.2}, we can say
	
\begin{equation*}  
    \sdefect({{ I}}, s)= 
    \begin{cases}
    &2k^2 - 1
    \vspace*{0.2cm}~\mbox{if}~  s = 2k \\
	&2k^2  +  2k   ~\mbox{if}~  s = 2k + 1     
	\end{cases}
	\end{equation*}	
	grows as a  quasi-polynomial in $ s $ with quasi-period $2.$
\end{example} 

\begin{remark}
Consider $G$ to be the odd cycle $C_3=(x_1,x_2,x_3).$ Let $ I(C_3)$ be the edge ideal of $C_3.$ Note that $C_3$ is a complete graph $K_3$ on vertices $x_1,x_2$ and $x_3$. Let $ J(K_3)$ be the cover ideal of $K_3.$	Here $I(C_3)=J(K_3).$
We have seen that the quasi-polynomial associated with the symbolic defects  of $ I(C_3)$  obtained by  our procedure as in Remark \ref{sd.unicyclic.2}  coincide with the  quasi-polynomial associated with the symbolic defects  of $ J(K_3)$    obtained by \cite[Theorem 5.7]{drabkin}.

\end{remark}

In the  next remark we see that the symbolic defects of edge ideals of two different  unicyclic graphs are the same under certain conditions.

\begin{remark}\label{remark}
 Let $ G_1 $  and  $G_2$  be  unicyclic graphs     with     unique  cycle  $ C_{2n+1}=(x_1,\ldots,x_{2n+1}). $  Let  $T$  and  $T^{\prime}$  are  the forests  $G_1  \setminus  E(C_{2n+1})$  and $G_2  \setminus  E(C_{2n+1})$,  respectively.  Let    $ {T}_r = V(C_{2n+1}) \cap V(T) = \{ y_1,y_2,\ldots,y_m \}  $  where each $ y_i = x_j $ for some  $ j \in [2n+1]$   and  $ {T}_r^{\prime} = V(C_{2n+1}) \cap V(T^{\prime}) = \{ z_1,z_2,\ldots,z_m \}  $  where each $ z_i = x_j $ for some  $ j \in [2n+1]$.  Let  $|E(T)|=l$  and  $|E(T^{\prime})|=l^{\prime}$.  Let $ l_1,l_2,\ldots,l_m $      be the number of edges connected to those $ m $ roots  $ y_1,y_2,\ldots,y_m   $ in   $ T,  $   respectively. Let $ l_1^{\prime},l_2^{\prime},\ldots,l_m^{\prime} $      be the number of edges connected to those $ m $ roots  $ z_1,z_2,\ldots,z_m   $ in   $ T^{\prime},  $   respectively.  If $l=l^{\prime}, l_1=l_1^{\prime} ,  l_2=l_2^{\prime},\ldots, l_m=l_m^{\prime}  $,  then by Theorem \ref{sd.unicyclic},  the symbolic defects of $I(G_1)$   and  $I(G_2)$  are same.  For example	consider $ G_1 $  and  $G_2$ to be the  unicyclic graphs as in Figure \ref{fig.3}. Note  that  $l = l^{\prime} = 4$ and  $l_1=l_1^{\prime}=2$.  Hence  the  symbolic defects of $I(G_1)$   and  $I(G_2)$  are same.   		    
 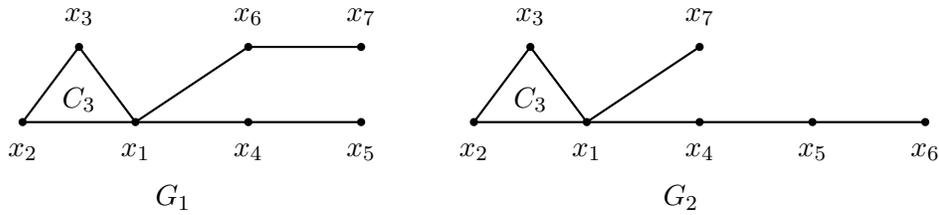
\begin{figure}[!ht]
 	\begin{tikzpicture}[scale=1]
 	\begin{scope}[ thick, every node/.style={sloped,allow upside down}] 
 	\definecolor{ultramarine}{rgb}{0.07, 0.04, 0.56} 
 	\definecolor{zaffre}{rgb}{0.0, 0.08, 0.66}
 	
 	\draw[fill][-,thick] (0,0) --(1.5,0);  
 	\draw[fill][-,thick](0,0) --(0.75,1);
 	\draw[fill][-,thick](0.75,1) --(1.5,0);
 	\draw[fill][-,thick]  (1.5,0) --(3,0);     
 	\draw[fill][-,thick]  (3,0) --(4.5,0);
 	\draw[fill][-,thick]  (1.5,0) --(3,1);     
 	\draw[fill][-,thick]  (3,1) --(4.5,1); 
 	
 	\draw[fill][-,thick] (0,0) circle [radius=0.04];
 	\draw[fill][-,thick] (1.5,0) circle [radius=0.04];
 	\draw[fill][-,thick] (0.75,1) circle [radius=0.04];
 	\draw[fill][-,thick] (3,0) circle [radius=0.04];
    \draw[fill][-,thick] (4.5,0) circle [radius=0.04];
 	\draw[fill][-,thick] (3,1) circle [radius=0.04];
 	\draw[fill][-,thick] (4.5,1) circle [radius=0.04];
 	
 	\node at (0,-0.4) {$x_2$};
 	\node at (1.5,-0.4) {$x_1$};
 	\node at (0.75,1.4) {$x_{3}$};  
 	\node at (3,-0.4) {$x_{4}$};    
 	\node at (4.5,-0.4) {$x_{5}$};
 	\node at (3,1.4) {$x_{6}$};
 	\node at (4.5,1.4) {$x_{7}$};
 	\node at (2,-1) {$G_1$};            
 	
 	\node at (0.75,0.3) {$C_3$};	
 	
 	\draw[fill][-,thick] (6,0) --(7.5,0);  
\draw[fill][-,thick](6,0) --(6.75,1);
\draw[fill][-,thick](6.75,1) --(7.5,0);
\draw[fill][-,thick]  (7.5,0) --(9,0);     
\draw[fill][-,thick]  (9,0) --(10.5,0);
\draw[fill][-,thick]  (7.5,0) --(9,1);     
\draw[fill][-,thick]  (10.5,0) --(12,0); 

\draw[fill][-,thick] (6,0) circle [radius=0.04];
\draw[fill][-,thick] (7.5,0) circle [radius=0.04];
\draw[fill][-,thick] (6.75,1) circle [radius=0.04];
\draw[fill][-,thick] (9,0) circle [radius=0.04];
\draw[fill][-,thick] (10.5,0) circle [radius=0.04];
\draw[fill][-,thick] (9,1) circle [radius=0.04];
\draw[fill][-,thick] (12,0) circle [radius=0.04];

\node at (6,-0.4) {$x_2$};
\node at (7.5,-0.4) {$x_1$};
\node at (6.75,1.4) {$x_{3}$};  
\node at (9,-0.4) {$x_{4}$};    
\node at (10.5,-0.4) {$x_{5}$};
\node at (9,1.4) {$x_{7}$};
\node at (12,-0.4) {$x_{6}$};
\node at (8.75,-1) {$G_2$};            

\node at (6.75,0.3) {$C_3$}; 	
 	
 	\end{scope}
 	\end{tikzpicture}
 	\caption{Two unicyclic graphs $G_1$ and $G_2$ with unique  cycle  $C_3$.}\label{fig.3}
 \end{figure}

\end{remark}

\section{Hilbert Function of Quotient module $I^{(s)}/I^s$}
Let $I \subset R=k[x_1,\ldots,x_n]$ be a homogeneous ideal. Let $\alpha(I)$ denote the degree of least degree monomial of $I$ and set $I_k$ to be the set of elements of $I$ of degree $k$.
Let $M = \displaystyle{\bigoplus_{i=0}^{\infty} M_i}$ be a finitely generated graded $R$-module	
	where $M_i$ denotes the $i$-th graded component of $M.$ Then the Hilbert function of the  $R$-module $M$ is defined as $$ h (d) = \dim_k M_d.$$
	Now consider  $I$ to be the  edge ideal of a unicyclic graph $G$ with a unique odd cycle.
In this section,    we give necessary and sufficient conditions on the structure of $G$   for some power of the maximal ideal to annihilate $ I^{(s)}/I^s$ and  also for those class of graphs, we compute the Hilbert function of the module $I^{(s)}/I^s$. 
\begin{theorem}\label{annhilator}
	Let $I$ be the edge ideal of a unicyclic graph  $G$   with a unique  cycle  $ C_{2n+1}=(x_1,\ldots,x_{2n+1}). $  Let $ s \in \mathbb{N} $ and write $ s = k(n + 1) + r $ for some $ k \in \mathbb{Z} $ and $ 0 \leq  r    \leq  n.$	Assume that $s \geq n+1.$ Let $ \m  $ be the maximal homogeneous ideal of $R$ generated by the elements of $V(G).$ Then $\m^{\gamma(s)} I^{(s)} \subseteq I^s$ for some  $ \gamma(s) $ if and only if  each vertex of $G$  lies in the set $N[V(C_{2n+1})].$  
\end{theorem}	
	\begin{proof}
 By Lemma \ref{unicyclic} for $ s = k(n + 1) + r  $ where $ k \in \mathbb{Z}  $ and $ 0 \leq  r    \leq  n, $      we have  $${I}^{(s)} =   \displaystyle{ \sum_{t=0}^{k}  {I}^{s-t(n+1)}  (c)^t  }~ \mbox{where}~ c= x_1 \cdots x_{2n+1} .$$		
 Assume that $\m^{\gamma(s)} I^{(s)} \subseteq I^s$ for some  $ \gamma(s) .$
Suppose there  exists a vertex $x_{i_j}$ of $G$ which does not lie in the set $N[V(C_{2n+1})].$ 
Let $ g = x_{i_j}^{\gamma(s)}(x_1x_2)^{s-k(n+1)}  c^k $.	 Then $ g \in  \m^{\gamma(s)}I^{s-k(n+1)}(c)^k$. Here  $ \deg(\frac{g}{ x_{i_j}^{\gamma(s)}}) = \deg((x_1x_2)^{s-k(n+1)}  c^k)    =2(s-k(n+1))+ k(2n+1) =  2s-k < 2s $ because $k \geq 1$  and  so $\frac{g}{ x_{i_j}^{\gamma(s)}} \notin I^s$. 
Note that each  variable  dividing $ \frac{g}{ x_{i_j}^{\gamma(s)}}  $  lies in  $V(C_{2n+1})$.  
Since $x_{i_j} \notin N[V(C_{2n+1})],$  $x_{i_j}$  can  not pair up with any element of  $V(C_{2n+1})$ to get a minimal generator of $I$. 
This implies  $x_{i_j}^{\gamma(s)}$ in $g$ remains unpaired.  Thus  $g = x_{i_j}^{\gamma(s)}. \frac{g}{ x_{i_j}^{\gamma(s)}} \notin I^s$  because  $\frac{g}{ x_{i_j}^{\gamma(s)}} \notin I^s$.
  Therefore    $ \m^{\gamma(s)}I^{s-k(n+1)}(c)^k \nsubseteq I^s.$ Hence $\m^{\gamma(s)} I^{(s)}  = \m^{\gamma(s)}(  \displaystyle{ \sum_{t=0}^{k}  {I}^{s-t(n+1)}  (c)^t  })  \nsubseteq I^s,$ i.e., a contradiction.
\vspace*{0.1cm}\\		
		Assume that each vertex of $G$  lies in the set $N[V(C_{2n+1})].$
   Let $u \in \mathcal{G}(\m c).$ Then $u= x_{i_1}c$ for some $x_{i_1} \in N[V(C_{2n+1})].$ Note that $cx_{i_1} \in \mathcal{G}(I^{n+1}) .$
Thus $ \m c \subseteq I^{n+1}.  $  
Therefore $ \m^{k}  {I}^{s-k(n+1)}  (c)^k  =  {I}^{s-k(n+1)} \m^{k}   (c)^k \subseteq {I}^{s-k(n+1)} I^{k(n+1)} = I^s. $        
Similarly for $1 \leq t \leq k-1, $ $ \m^{k}  {I}^{s-t(n+1)}  (c)^t  =  \m^{k-t}  {I}^{s-t(n+1)} \m^{t}   (c)^t \subseteq \m^{k-t}{I}^{s-t(n+1)} I^{t(n+1)} = \m^{k-t}I^s   \subseteq I^s. $      
This implies that  $ \m^{k}  {I}^{(s)}  = \m^{k}(  \displaystyle{ \sum_{t=0}^{k}  {I}^{s-t(n+1)}  (c)^t  }) \subseteq  I^s. $ Hence $\m^{\gamma(s)} I^{(s)} \subseteq I^s$ for $ \gamma(s) = k .$
\end{proof}   	

In the following remark, we answer question (2) asked by Herzog as noted in the introduction if $I$ is the edge ideal of a unicyclic graph with a unique odd cycle.    
%
\begin{remark}
Let $I$ be the edge ideal of a unicyclic graph $G$ with a unique  cycle  $ C_{2n+1}=(x_1,\ldots,x_{2n+1}) $. Let $ s \in \mathbb{N} $ and write $ s = k(n + 1) + r $ for some $ k \in \mathbb{Z} $ and $ 0 \leq  r    \leq  n.$	Assume that $s \geq n+1.$ Let $ \m  $ be the maximal homogeneous ideal of $R$ generated by the elements of $V(G).$	
By Lemma \ref{unicyclic} for $ s = k(n + 1) + r  $ where $ k \in \mathbb{Z}  $ and $ 0 \leq  r    \leq  n, $      we have  $${I}^{(s)} =   \displaystyle{ \sum_{t=0}^{k}  {I}^{s-t(n+1)}  (c)^t  }~ \mbox{where}~ c= x_1 \cdots x_{2n+1} .$$  
 If there exists some vertex of  $G$ which does not lie  in the set $N[V(C_{2n+1})],$ then by Lemma \ref{annhilator},   
$\m^{\gamma(s)} I^{(s)} \nsubseteq I^s$ for any $ \gamma(s) $. Hence no powers of maximal ideal can be  annihilator of $ I^{(s)}/I^s$ in this case.\\
Now we assume that $ G $ is a unicyclic graph     with a unique  cycle  $ C_{2n+1}=(x_1,\ldots,x_{2n+1}) $ where    each vertex of $G$  lies in the set $N[V(C_{2n+1})]$. By Lemma \ref{annhilator},   
$\m^{k} I^{(s)} \subseteq I^s$. Let $ g = x_{1}^{k-1}(x_1x_2)^{s-k(n+1)}  c^k $. Notice that $ g \in  \m^{k-1}I^{s-k(n+1)}(c)^{k}$. 		Here $ \deg(g) = k-1 + 2(s-k(n+1)) + k(2n+1) = 2s-1$. So   $g \notin I^s$.  
Thus    $ \m^{k-1}I^{s-k(n+1)}(c)^k \nsubseteq I^s.$ This implies $\m^{k-1} I^{(s)}  = \m^{k-1}(  \displaystyle{ \sum_{t=0}^{k}  {I}^{s-t(n+1)}  (c)^t  })  \nsubseteq I^s$.
Therefore $k$ is the least integer such that $\m^{k} I^{(s)} \subseteq I^s.$ Hence 
$ \m ^{ \gamma(s)}. (I^{(s)}/I^s) = 0 $ where $\gamma(s) \geq k$ and  $\gamma(s) $ increases as $s$ increases.
\end{remark}    

\begin{lemma}\label{degree2s}
	
	Let $I$ be the edge ideal of a unicyclic graph $G$     with a unique  cycle  $ C_{2n+1}=(x_1,\ldots,x_{2n+1}) $ where    each vertex of $G$  lies in the set $N[V(C_{2n+1})]$. Let $ s \in \mathbb{N} $ and write $ s = k(n + 1) + r $ for some $ k \in \mathbb{Z} $ and $ 0 \leq  r    \leq  n.$   Then $ I^{s}_{2s}  =    I^{(s)}_{2s}  $.

\end{lemma}    

\begin{proof}
	Since $I^s \subseteq I^{(s)},$ it follows that    $ I^{s}_{2s} \subseteq       I^{(s)}_{2s}  .$\\
	By Lemma \ref{unicyclic} for $ s = k(n + 1) + r  $ where $ k \in \mathbb{Z}  $ and $ 0 \leq  r    \leq  n, $      we have  $${I}^{(s)} =   \displaystyle{ \sum_{t=0}^{k}  {I}^{s-t(n+1)}  (c)^t  }~ \mbox{where}~ c= x_1 \cdots x_{2n+1} .$$  
	
	Let $ g \in I^{(s)}_{2s}.$ Then $ g \in  {I}^{s-t(n+1)}  (c)^t$ for some $t$ where $0 \leq t \leq k.$ Here $ \deg(g) = 2s =  2 (s-t(n+1)) +   t(2n+1)  + t.$ So  $ g = hc^tx_{i_1} \cdots x_{i_t}$
	for some $h \in \mathcal{G}({I}^{s-t(n+1)})$ and $x_{i_j} \in N[V(C_{2n+1})]$ for $1 \leq j \leq t.$
	Note that $cx_{i_{j}}$  $\in \mathcal{G}(I^{n+1})$ for   $1 \leq j \leq t.$     
	Then $ c^t{x_{i_1}\cdots   x_{i_{t}}}
	=(c{x_{i_{1}}})\cdots (c{x_{i_{t}}})  $
	$ 	 \in \mathcal{G}({I^{t(n+1)}}) .$
	Thus
	$ g = hc^t{x_{i_1}\cdots   x_{i_{t}}} $
	$ 	 \in \mathcal{G}(I^{s-t(n+1)}{I^{t(n+1)}})=\mathcal{G}(I^s) .$ Hence $g \in I^{s}_{2s}.$ 
\end{proof}

\begin{theorem}\label{hilbert.thm}
	Let $I$ be the edge ideal of a unicyclic graph $G$     with a unique  cycle  $ C_{2n+1}=(x_1,\ldots,x_{2n+1}) $ where    each vertex of $G$  lies in the set $N[V(C_{2n+1})]$. Let $|E(G) \setminus E(C_{2n+1})| = l.$   Let $ s \in \mathbb{N} $ and write $ s = k(n + 1) + r $ for some $ k \in \mathbb{Z} $ and $ 0 \leq  r    \leq  n.$ Assume that $s \geq n+1.$
	Then the Hilbert function of the module $I^{(s)}/I^s$ is given by 
 $\displaystyle{h(d = 2s-k + i ) = \binom{2n+l + s-(k-i)(n+1) }{s-(k-i)(n+1)}}$   for $0 \leq i \leq k-1,$\\
 $h(d)=0$ otherwise.

\end{theorem}

\begin{proof}
	Since $|E(G) \setminus E(C_{2n+1})| = l,$  $|E(G)| = 2n + 1 + l.$ By Lemma \ref{unicyclic} for $ s = k(n + 1) + r  $ where $ k \in \mathbb{Z}  $ and $ 0 \leq  r    \leq  n, $      we have  $${I}^{(s)} =   \displaystyle{ \sum_{t=0}^{k}  {I}^{s-t(n+1)}  (c)^t  }~ \mbox{where}~ c= x_1 \cdots x_{2n+1} .$$
	Here $I^{(s)}/I^s \cong    \displaystyle{ \sum_{t=1}^{k}  {I}^{s-t(n+1)}  (c)^t  }(\mbox{mod}~ I^s).$ Since $\deg(c) = 2n + 1,$ $\alpha(I^{s-t(n+1)}(c)^t)= 2 (s-t(n+1)) + t(2n+1) = 2s-t $ for $1 \leq t \leq k.$
	If $t_1 \leq t_2,$ then 
	$\alpha(I^{s-t_2(n+1)}(c)^{t_2}) \leq\alpha(I^{s-t_1(n+1)}(c)^{t_1}) .$ Thus $\alpha(I^{(s)}/I^s) = \alpha(I^{s-k(n+1)}(c)^k) = 2s-k .$ So the Hilbert function of  $I^{(s)}/I^s$ is zero in degrees    $ < \alpha(I^{(s)}/I^s)= 2s-k .$
	Also by Lemma \ref{degree2s},   
	the Hilbert function of  $I^{(s)}/I^s$ is zero in degrees  $\geq  2s.$
	\vspace*{0.1cm}\\
Let $ g \in [{I}^{s-k(n+1)}  (c)^k]_{ 2s-k +r} $  for some $r \geq 0$ such that $g \notin I^s.$ Here $ \deg(g) = 2s-k +r =  2 (s-k(n+1)) +   k(2n+1)  + r.$ So $g = h c^k x_{i_1}\cdots x_{i_r}$ where $h \in \mathcal{G}(I^{s-k(n+1)})$ and each $x_{i_j} \in N[V(C_{2n+1})].$ Note that $cx_{i_j} \in \mathcal{G}({I}^{n+1})$ for  $1 \leq j \leq r.$
	If $r \geq k,$ then 
	\begin{align*}
	g &= h c^k (x_{i_1}\cdots x_{i_k}) (x_{i_{k+1}}\cdots x_{i_r} )\\
	&= h [(c x_{i_1})\cdots (cx_{i_k})] (x_{i_{k+1}}\cdots x_{i_r} )\\
	&  \in {I}^{s-k(n+1)}{I}^{k(n+1)} (x_{i_{k+1}}\cdots x_{i_r} ) \subseteq I^s,~\mbox{	i.e., a contradiction.}
	\end{align*}
 
	So $r < k.$ Then 
	$g = h c^k   (x_{i_1}\cdots x_{i_r})
	= h c^r   (x_{i_1}\cdots x_{i_r})c^{k-r}= h [(cx_{i_1})\cdots (cx_{i_r})]c^{k-r} .$  
	Since $ h \in   \mathcal{G}(I^{s-k(n+1)}) $ and 
	$(cx_{i_1})\cdots (cx_{i_r}) \in \mathcal{G}({I}^{r(n+1)})  ,$ 
	$h [(c x_{i_1})\cdots (cx_{i_k})]   \in \mathcal{G}({I}^{s-k(n+1)}I^{r(n+1)}) =  \mathcal{G}({I}^{s-(k-r)(n+1)}).$ Thus $ g  \in    \mathcal{G}({I}^{s-(k-r)(n+1)}(c)^{k-r})$
	where $ 1 \leq k-r  \leq k.$
	\vspace*{0.1cm}\\
Therefore  multiple of any element of the set $  [{I}^{s-k(n+1)}  (c)^k]_{ 2s-k }  $ of degree $\geq  2s-k $ not lying in $I^s$ must lie in some set $ \mathcal{G}({I}^{s-t(n+1)}(c)^{t}) $  where $1 \leq t \leq k.$
	\vspace*{0.1cm}\\
	Similarly, the multiple of any element of the set $  [{I}^{s-p(n+1)}  (c)^p]_{2s-p }  $  of degree $\geq  2s-p $ for  $1 \leq p \leq k-1$ not lying   in $I^s$ must lie in   some    $ \mathcal{G}({I}^{s-t(n+1)}(c)^{t})  $  where $1 \leq t \leq p.$ Thus the set of elements of ${I}^{(s)}/I^{s}$ of degree $ 2s-t  $ is $ \mathcal{G}({I}^{s-t(n+1)}(c)^{t})~(\mbox{mod } I^s)    $ for  $1 \leq t \leq k,$  
i.e.,
	$[{I}^{(s)}/I^{s}]_{2s-(k - i) }    
	= \mathcal{G}(I^{s-(k-i)(n+1)} (c)^{(k-i)})~(\mbox{mod } I^s)$ for $ 0\leq i\leq k-1.$
%
%
	 Hence by Remark \ref{unique3} for $0 \leq i \leq k-1$,  we have
	\begin{align*}
	h( 2s-k + i) &= h(2s-(k - i))\\
	&= \big|[{I}^{(s)}/I^{s}]_{ 2s-(k - i)}\big|\\    
	&= \big|\mathcal{G}(I^{s-(k-i)(n+1)} (c)^{(k-i)})\big|\\
	&= \big|\mathcal{G}(I^{s-(k-i)(n+1)} )\big|\\
	&=\multiset{2n+1+l}{s-(k-i)(n+1)}\\
	&=  \binom{2n+l + s-(k-i)(n+1) }{s-(k-i)(n+1)}. 
	\end{align*}

\end{proof}

\begin{remark}
	 Let $ G_1 $  and  $G_2$  be  unicyclic graphs     with     unique  cycle  $ C_{2n+1}=(x_1,\ldots,x_{2n+1}). $ Let $I_1 = I(G_1)$ and  $I_2 = I(G_2)$.  Let  $T$  and  $T^{\prime}$  are  the forests  $G_1  \setminus  E(C_{2n+1})$  and $G_2  \setminus  E(C_{2n+1})$,  respectively.    Let  $|E(T)|=l$  and  $|E(T^{\prime})|=l^{\prime}$.   If $l=l^{\prime}$, then by Theorem \ref{hilbert.thm},  the Hilbert  functions of ${I_1}^{(s)}/{I_1}^{s}$   and  ${I_2}^{(s)}/{I_2}^{s}$  are equal for $s \geq n+1$.  For example	consider $ G_1 $  and  $G_2$ to be the  unicyclic graphs  in Figure \ref{fig.4}. Note  that  $l = l^{\prime} = 4$.  Hence  the  Hilbert  functions of ${I_1}^{(s)}/{I_1}^{s}$   and  ${I_2}^{(s)}/{I_2}^{s}$  are same for $s \geq 2$.   		    
	\begin{figure}[!ht]
		\begin{tikzpicture}[scale=1]
		\begin{scope}[ thick, every node/.style={sloped,allow upside down}] 
		\definecolor{ultramarine}{rgb}{0.07, 0.04, 0.56} 
		\definecolor{zaffre}{rgb}{0.0, 0.08, 0.66}
		
		\draw[fill][-,thick] (0,0) --(1.5,0);  
		\draw[fill][-,thick](0,0) --(0.75,1);
		\draw[fill][-,thick](0.75,1) --(1.5,0);
		\draw[fill][-,thick]  (1.5,0) --(3,0);     
		\draw[fill][-,thick]  (3,0) --(4.5,0);
		\draw[fill][-,thick]  (0.75,1) --(2.25,1);     
		\draw[fill][-,thick]  (2.25,1) --(3.75,1); 
		
		\draw[fill][-,thick] (0,0) circle [radius=0.04];
		\draw[fill][-,thick] (1.5,0) circle [radius=0.04];
		\draw[fill][-,thick] (0.75,1) circle [radius=0.04];
		\draw[fill][-,thick] (3,0) circle [radius=0.04];
		\draw[fill][-,thick] (4.5,0) circle [radius=0.04];
		\draw[fill][-,thick] (2.25,1) circle [radius=0.04];
		\draw[fill][-,thick] (3.75,1) circle [radius=0.04];
		
		\node at (0,-0.4) {$x_2$};
		\node at (1.5,-0.4) {$x_1$};
		\node at (0.75,1.4) {$x_{3}$};  
		\node at (3,-0.4) {$x_{4}$};    
		\node at (4.5,-0.4) {$x_{5}$};
		\node at (2.25,1.4) {$x_{6}$};
		\node at (3.75,1.4) {$x_{7}$};
		\node at (2,-1) {$G_1$};            
		
		\node at (0.75,0.3) {$C_3$};	
		
		\draw[fill][-,thick] (6,0) --(7.5,0);  
		\draw[fill][-,thick](6,0) --(6.75,1);
		\draw[fill][-,thick](6.75,1) --(7.5,0);
		\draw[fill][-,thick]  (7.5,0) --(9,0);     
		\draw[fill][-,thick]  (9,0) --(10.5,0);
		\draw[fill][-,thick]  (7.5,0) --(9,1);     
		\draw[fill][-,thick]  (10.5,0) --(12,0); 
		
		\draw[fill][-,thick] (6,0) circle [radius=0.04];
		\draw[fill][-,thick] (7.5,0) circle [radius=0.04];
		\draw[fill][-,thick] (6.75,1) circle [radius=0.04];
		\draw[fill][-,thick] (9,0) circle [radius=0.04];
		\draw[fill][-,thick] (10.5,0) circle [radius=0.04];
		\draw[fill][-,thick] (9,1) circle [radius=0.04];
		\draw[fill][-,thick] (12,0) circle [radius=0.04];
		
		\node at (6,-0.4) {$x_2$};
		\node at (7.5,-0.4) {$x_1$};
		\node at (6.75,1.4) {$x_{3}$};  
		\node at (9,-0.4) {$x_{4}$};    
		\node at (10.5,-0.4) {$x_{5}$};
		\node at (9,1.4) {$x_{7}$};
		\node at (12,-0.4) {$x_{6}$};
		\node at (8.75,-1) {$G_2$};            
		
		\node at (6.75,0.3) {$C_3$}; 	
		
		\end{scope}
		\end{tikzpicture}
		\caption{Two unicyclic graphs $G_1$ and $G_2$ with unique  cycle  $C_3$.}\label{fig.4}
	\end{figure}
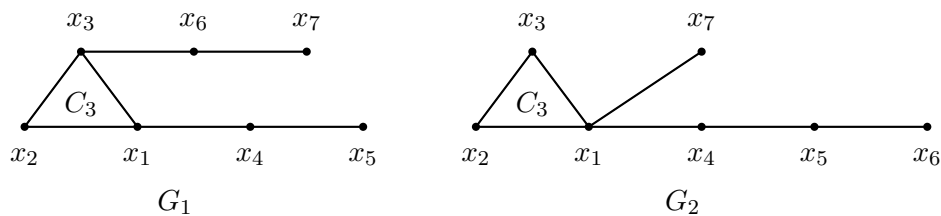

\end{remark}

\textbf{Acknowledgements}
\vspace*{0.15cm}\\
  We would like to thanks the anonymous referee for his/her helpful  suggestions and pointing out the example in Remark \ref{remark}. The first author was supported by SERB (grant No.: EMR/2016/006997), India.

\end{document}